\documentclass{amsart}
\usepackage{graphicx}
\usepackage{amssymb}
\newtheorem{thm}{Theorem}[section]
\newtheorem{cor}[thm]{Corollary}
\newtheorem{lem}[thm]{Lemma}
\newtheorem{prop}[thm]{Proposition}
\theoremstyle{definition}
\newtheorem{defn}[thm]{Definition}
\newtheorem{eg}[thm]{Example}
\theoremstyle{remark}
\newtheorem{rem}[thm]{Remark}
\newtheorem*{acknowledgement}{Acknowledgements}
\numberwithin{equation}{section}

\newcommand{\hf}{\frac{1}{2}}
\newcommand{\thd}{\frac{1}{3}}

\newcommand{\sh}{\mathrm{sh}}
\newcommand{\cl}{\mathrm{cl}}
\newcommand{\Tw}{\mathrm{Tw}}
\newcommand{\RR}{\mathbb{R}}

\newcommand{\QQ}{\mathbb{Q}}
\newcommand{\ZZ}{\mathbb{Z}}

\newcommand{\To}{\longrightarrow}

\newcommand{\A}{\mathcal{A}}
\newcommand{\KK}{\mathbb{K}}
\newcommand{\K}{\mathcal{K}}
\renewcommand{\L}{\mathcal{L}}
\newcommand{\R}{\mathcal{R}}
\newcommand{\E}{\mathcal{E}}
\newcommand{\C}{\mathcal{C}}
\newcommand{\D}{\mathcal{D}}
\newcommand{\M}{\mathcal{M}}
\newcommand{\I}{\mathcal{I}}
\newcommand{\F}{\mathcal{F}}
\newcommand{\Q}{\mathcal{Q}}
\renewcommand{\O}{\mathcal{O}}
\newcommand{\X}{\mathfrak{X}}
\newcommand{\g}{\mathfrak{g}}
\newcommand{\ug}{\underline{\mathfrak{g}}}
\newcommand{\T}{\mathbb{T}}
\newcommand{\metr}{\langle\cdot,\cdot\rangle}
\newcommand{\met}[2]{\langle#1,#2\rangle}
\newcommand{\brac}{[\cdot,\cdot]}
\newcommand{\cb}[2]{[\![#1,#2]\!]}
\newcommand{\cbr}{[\![\cdot,\cdot]\!]}

\newcommand{\pbr}{\{\cdot,\cdot\}}
\newcommand{\pb}[2]{\{#1,#2\}}
\newcommand{\rbr}{(\cdot,\cdot)}
\newcommand{\dl}{\partial}
\newcommand{\Hom}{\mathrm{Hom}}
\newcommand{\Der}{\mathrm{Der}}
\newcommand{\gr}{\mathrm{gr}}
\begin{document}

\title[Courant-Dorfman algebras]{Courant-Dorfman algebras and their cohomology}%
\author{Dmitry Roytenberg}%
\address{Max Planck Institut f\"{u}r Mathematik, Vivatsgasse 7, 53111 Bonn}%
\email{roytenbe@mpim-bonn.mpg.de}%


\date{\today}%
\dedicatory{To the memory of I.Ya. Dorfman (1948 - 1994)}%
\begin{abstract}
We introduce a new type of algebra, the Courant-Dorfman algebra.
These are to Courant algebroids what Lie-Rinehart algebras are to
Lie algebroids, or Poisson algebras to Poisson manifolds. We work
with arbitrary rings and modules, without any regularity, finiteness
or non-degeneracy assumptions. To each Courant-Dorfman algebra
$(\R,\E)$ we associate a differential graded algebra $\C(\E,\R)$ in
a functorial way by means of explicit formulas. We describe two
canonical filtrations on $\C(\E,\R)$, and derive an analogue of the
Cartan relations for derivations of $\C(\E,\R)$; we classify central
extensions of $\E$ in terms of $H^2(\E,\R)$ and study the canonical
cocycle $\Theta\in\C^3(\E,\R)$ whose class $[\Theta]$ obstructs
re-scalings of the Courant-Dorfman structure. In the nondegenerate
case, we also explicitly describe the Poisson bracket on
$\C(\E,\R)$; for Courant-Dorfman algebras associated to Courant
algebroids over finite-dimensional smooth manifolds, we prove that
the Poisson dg algebra $\C(\E,\R)$ is isomorphic to the one
constructed in \cite{Roy4-GrSymp} using graded manifolds.
\end{abstract}
\maketitle

\section{Introduction}
\subsection{Historical background}
This is the first in a series of papers devoted to the study of
Courant-Dorfman algebras. These algebraic structures first arose in
the work of Irene Dorfman \cite{Dorf87} and Ted Courant \cite{Cou} on
reduction in classical mechanics and field theory (with \cite{CouWe}
a precursor to both, ultimately leading back to \cite{Dir}). Courant
considered sections of the vector bundle $\mathbf{T}M=TM\oplus T^*M$
over a finite-dimensional $C^\infty$ manifold $M$, endowed with the
canonical pseudo-metric
$$
\met{(v_1,\alpha_1)}{(v_2,\alpha_2)}=\iota_{v_1}\alpha_2+\iota_{v_2}\alpha_1
$$
and a new bracket he introduced:
$$
\cb{(v_1,\alpha_1)}{(v_2,\alpha_2)}=(\pb{v_1}{v_2},L_{v_1}\alpha_2-L_{v_2}\alpha_1-\hf
d_0(\iota_{v_1}\alpha_2-\iota_{v_2}\alpha_1)),
$$
while Dorfman was working in a more general abstract setting
involving a Lie algebra $\X^1$ and a complex $\Omega$ acted upon by
the differential graded Lie algebra $T[1]\X^1=\X^1[1]\oplus\X^1$
(i.e., to each $v\in\X^1$ there are associated operators $\iota_v$
and $L_v$ on $\Omega$ satisfying the usual Cartan relations). She
considered the space $\Q=\X^1\oplus\Omega^1$ equipped with the above
pseudo-metric and a bracket given by
$$
[(v_1,\alpha_1),(v_2,\alpha_2)]=(\pb{v_1}{v_2},L_{v_1}\alpha_2-\iota_{v_2}d_0\alpha_1)
$$
Here $\pbr$ denotes the commutator of vector fields (resp. the
bracket on $\X^1$), while $d_0$ denotes the exterior derivative
(resp. the differential on $\Omega$). In both cases, a \emph{Dirac
structure} was defined to be a subbundle $D\subset\mathbf{T}M$
(resp. a subspace $\D\subset\Q$) which is maximally isotropic with
respect to $\metr$ and closed under the Courant (resp. Dorfman)
bracket; each Dirac structure defines a Poisson bracket on a
subalgebra of $C^\infty(M)$ (resp. a subspace of $\Omega^0$), thus
explaining the role of this formalism in the theory of constrained
dynamical systems, both in mechanics and field theory. We refer to \cite{Dorf} for
an excellent exposition of these ideas.

The notion of a \emph{Courant algebroid} was introduced in
\cite{LWX1} where it was used to generalize the theory of Manin
triples to Lie bialgebroids; it involved a vector bundle equipped
with a pseudo-metric, a Courant bracket and an anchor map to the
tangent bundle, satisfying a set of compatibility conditions. The
notion has since turned up in other contexts. \v{S}evera
\cite{Sev-Letters} discovered that a Courant (or Dorfman) bracket
could be twisted by a closed 3-form, as a result of which the
Courant algebroid $\mathbf{T}M$ took the place of $TM$ in Hitchin's
``generalized differential geometry" (i.e. differential geometry in the
presence of an abelian gerbe \cite{Hit}); he also noted that
transitive Courant algebroids could be used to give an
obstruction-theoretic interpretation of the first Pontryagin class
(this theory was fully worked out by Bressler \cite{Bres}, who also
elucidated the relation with vertex operator algebras). In general,
there is mounting evidence that Courant algebroids play the same
r\^{o}le in string theory as Poisson structures do in particle
mechanics \cite{Sev-Letters, Bres, AStr}.

\subsection{The aim and content of this paper}
In our earlier work \cite{Roy1, Roy4-GrSymp} we made an attempt to
explain Courant algebroids in terms of graded differential geometry
by constructing, for each vector bundle $E$ with a non-degenerate
pseudo-metric, a graded symplectic (super)manifold $M(E)$. We proved
that Courant algebroid structures on $E$ correspond to functions
$\Theta\in\C^3(M(E))$ obeying the Maurer-Cartan equation
$$
\pb{\Theta}{\Theta}=0
$$
The advantage of this approach is geometric clarity: after all,
graded manifolds are just manifolds with a few bells and whistles,
and our construction uses nothing more than a cotangent bundle. As a
by-product, it yielded new examples of topological sigma-models
\cite{Roy5-AKSZ}. Moreover, the graded manifold approach enabled
\v{S}evera \cite{Sev2} to envision an infinite hierarchy of graded
symplectic structures similar to the hierarchy of higher categories
(our construction in \cite{Roy4-GrSymp} is equivalent to a special
case of his).

Nevertheless, the formulation in terms of graded manifolds has
certain drawbacks. In particular, we were unable to describe the
algebra of functions $\C(M(E))$ explicitly in terms of $E$, which
made it somewhat difficult to work with: general considerations
(such as grading) would carry one a certain distance, but to go
beyond that, one had to either resort to local coordinates, or
introduce unnatural extra structure, such as a connection, which
rather spoiled the otherwise beautiful picture.

The aim of this paper is to obtain a completely explicit description
of the algebra $\C(M(E))$. We work from the outset with a
commutative algebra $\R$ and an $\R$-module $\E$ equipped with a
pseudo-metric $\metr$; these can be completely arbitrary: no
regularity or finiteness conditions are imposed on $\R$ or $\E$, nor
is $\metr$ required to be non-degenerate\footnote{This is still more
than Dorfman \cite{Dorf87} required: what she was dealing with is an
example of a structure we called \emph{hemi-strict Lie 2-algebra} in
\cite{Roy6-WL2}.}. A \emph{Courant-Dorfman algebra} consists of this
underlying structure, plus an $\E$-valued derivation $\dl$ of $\R$ and a
(Dorfman) bracket $\brac$, satisfying compatibility conditions
generalizing those defining a Courant algebroid.

Given a metric $\R$-module $(\E,\metr)$, we construct a graded
commutative $\R$-algebra $\C(\E,\R)$ whose degree-$q$ component
consists of (finite) sequences\\
$\omega=(\omega_0,\omega_1,\ldots)$, where each $\omega_k$ is an
$\R$-valued function of $q-2k$ arguments from $\E$ and $k$ arguments
from $\R$. With respect to the $\R$-arguments, $\omega_k$ is a
symmetric $k$-derivation; the behavior of $\omega_k$ under
permutations of the $\E$-arguments and the multiplication of these arguments by elements of $\R$ is
controlled by $\omega_{k+1}$. The algebra $\C(\E,\R)$ is actually a
subalgebra of the convolution algebra $\Hom(U(L),\R)$ where $L$ is a
certain graded Lie algebra.

Furthermore, every Courant-Dorfman structure on the metric module
$\E$ gives rise to a differential on the algebra $\C(\E,\R)$ for
which we give an explicit formula \eqref{eqn:diff}. The construction
is functorial with respect to (strict) morphisms of Courant-Dorfman
algebras; this is the content of our first main Theorem
\ref{thm:CD->DGA}. The resulting cochain complex, which we call the
\emph{standard complex}, is related to the Loday-Pirashvili complex
\cite{LodPir} for the Leibniz algebra $(\E,\brac)$ in a way
analogous to how the de Rham complex of a manifold is related to the
Chevalley-Eilenberg complex of its Lie algebra of vector fields.

We then conduct further investigation of the differential graded
algebra $\C(\E,\R)$. In particular, we describe natural filtrations
and subcomplexes, related to those considered in \cite{StXu} and
\cite{GinGru}, which we expect to be an important tool in cohomology
computations; derive commutation relations among certain derivations
of $\C(\E,\R)$, similar to the well-known Cartan relations among
contractions and Lie derivatives by vector fields; classify central
extensions of the Courant-Dorfman algebra $\E$ in terms of
$H^2(\E,\R)$. We also consider the canonical cocycle
$\Theta=(\Theta_0,\Theta_1)\in\C^3(\E,\R)$ given by the formula:
\begin{eqnarray*}
\Theta_0(e_1,e_2,e_3)&=&\met{[e_1,e_2]}{e_3}\\
\Theta_1(e;f)&=&-\rho(e)f
\end{eqnarray*}
generalizing the Cartan 3-form on a quadratic Lie algebra appearing
in the Chern-Simons theory.

When the pseudo-metric $\metr$ is non-degenerate, the algebra
$\C(\E,\R)$ has a Poisson bracket for which we also give an explicit
formula (\eqref{eqn:bullet}, \eqref{eqn:diamond} and
\eqref{eqn:pb}); the differential is then Hamiltonian: $d=-\pb{\Theta}{\cdot}$ (Theorem \ref{thm:pb}).
Finally, for Courant-Dorfman algebras coming from finite-dimensional vector
bundles, we prove (Theorem \ref{thm:isom_gr_man}) that the
differential graded Poisson algebra $\C(\E,\R)$ is isomorphic to the
algebra $\C(M(E))$ constructed in \cite{Roy4-GrSymp}. The
isomorphism associates to every $\omega\in\C^p(M(E))$ the sequence
$\Phi\omega=((\Phi\omega)_0,(\Phi\omega_1),\ldots)\in\C^p(\E,\R)$
where
$$
(\Phi\omega)_k(e_1,\ldots,e_{p-2k};f_1,\ldots,f_k)=
$$
$$
=(-1)^{\frac{(p-2k)(p-2k-1)}{2}}\{\cdots\{\omega,e_1^\flat\},\cdots\},e_{p-2k}^\flat\},f_1\},\cdots\},f_k\}
$$
where $e^\flat=\met{e}{\cdot}$. Under this isomorphism, our
canonical cocycle $\Theta$ corresponds to the one constructed in
\emph{loc. cit.} In fact, this formula is the main creative input
for this work: all the other formulas were ``reverse-engineered" from
this one and then shown to be valid in the general case. This construction
should be compared to Voronov's ``higher derived brackets" \cite{Vor5}.

We would like to emphasize that, apart from overcoming the drawbacks
of the graded manifold formulation mentioned above and being
completely explicit, our constructions apply in a much more general
setting where extra structures, such as local coordinates or
connections, may not be available. One point worth mentioning is
that the algebra $\C(\E,\R)$ is generally \emph{not} freely
generated over $\R$ (in a sense which we hope to eventually make
precise, it is as free as possible in the presence of $\metr$);
rather, it has a filtration such that the associated graded algebra
is free graded commutative over $\R$. The situation is the following:
when $\R$ and $\E$ satisfy some finiteness
conditions and $\metr$ is non-degenerate, the set of isomorphisms of
$\C(\E,\R)$ with the free algebra
$\gr\C(\E,\R)=S_\R(\X^1[-2]\oplus\E^\vee[-1])$ is in 1-1
correspondence with the set of splittings of the extension of
$\R$-modules
$$
\Lambda^2_\R\E^\vee\rightarrowtail\C^2(\E,\R)\twoheadrightarrow\Der(\R,\R)=:\X^1
$$
This is known as an \emph{Atiyah sequence}; its splitting is nothing
but a metric connection on $\E$. The set of splittings may be empty,
but when $\R$ is ``smooth" (in the sense that the module
$\Der(\R,\R)$ is projective), splittings do exist and form a torsor
under $\Omega^1\otimes_\R\Lambda^2_\R\E^\vee$; nevertheless, it is
important to keep in mind that, when working with a smooth scheme or
a complex manifold, such splittings generally exist only locally,
while a global splitting is obstructed by the Atiyah class. In
\cite{Roy4-GrSymp} we wrote down the Poisson bracket on
$S_\R(\X^1[-2]\oplus\E^\vee[-1])$ corresponding to the canonical one
on $\C(M(E))$ under a given metric connection $\nabla$; we have
since been informed that this bracket had been known to physicists
under the name of \emph{Rothstein bracket} \cite{Roth}. For an approach using
this formulation we refer to \cite{KelWald}; our work here was
motivated by the desire to avoid any unnatural choices.

\subsection{The sequel(s)} We plan to write (at least) two sequels to
this paper, in which we address several issues not covered here. In
the first one, we introduce a closed 2-form $\Xi$ on the algebra
$\C(\E,\R)$; it corresponds to the one we constructed on $\C(M(E))$
in \cite{Roy4-GrSymp} (even for degenerate $\metr$). We use this
extra structure to, on the one hand, restrict the class of morphisms
of differential graded algebras to those which also preserve this
structure, and on the other hand, to expand the class of morphisms
of Courant-Dorfman algebras to include lax morphisms, so as to make
the functor from Theorem \ref{thm:CD->DGA} fully faithful. We will
also consider morphisms of Courant-Dorfman algebras over different
base rings. Furthermore, the 2-form $\Xi$ gives rise to a Poisson
bracket on a certain subalgebra of $\C(\E,\R)$ by a graded version
of Dirac's formalism \cite{Dir}.

In the second sequel, we consider the general notion of a module
over a Courant-Dorfman algebra, based on the notion of a dg module
over the dg algebra $\C(\E,\R)$ (possibly with some extra conditions
involving $\Xi$), and study the (derived) category of these modules.
One such module is the \emph{adjoint module} $\C(\E,\E)$ consisting
of derivations of $\C(\E,\R)$ preserving $\Xi$. It forms a
differential graded Lie algebra under the commutator bracket; this
dg Lie algebra controls the deformation theory of Courant-Dorfman
structures on a fixed underlying metric module, and is analogous to
the dg Lie algebra controlling deformations of Lie-Rinehart
structures on a fixed underlying module, described in \cite{CraMoe}.

Eventually, we hope to be able to re-write the whole story using an
approach involving nested operads.

\subsection{On relation with other work and choice of terminology}
Our definition of a Courant-Dorfman algebra is very similar to
Weinstein's ``$(R,\A) C$-algebras" \cite{We5}, except for his
non-degeneracy assumption and use of Courant, rather than Dorfman,
brackets. Keller and Waldmann \cite{KelWald} gave an ``algebraic"
definition of a Courant algebroid, while still retaining the
finiteness, regularity and non-degeneracy assumptions, and obtained
formulas similar to some of those derived here. Our description of
the algebra $\C(\E,\R)$ has the same spirit as the formulas
describing exterior powers of adjoint and co-adjoint representations
of a Lie algebroid in (\cite{CamMar}, Example 3.26 and subsection
4.2).

We feel justified in our choice of the term ``Courant-Dorfman
algebra": not only is it natural and easy to remember, but it also
recognizes the contributions of two mathematicians (one of whom has
long since left active research while the other is, sadly, no longer
with us) to the subject that has since grown in scope far beyond
what they had envisioned.

\subsection{Organization of the paper} The paper is organized as
follows. In Section \ref{sec:def} we define Courant-Dorfman
algebras, derive some of their basic properties and give a number of
examples of these structures, emphasizing connection with the
various areas of mathematics where they arise; Section
\ref{sec:conv} is devoted to the preliminary construction of a
convolution algebra associated to a graded Lie algebra; Section
\ref{sec:complex} is the heart of the paper, where we construct the
differential graded algebra $\C(\E,\R)$ and study its properties;
Section \ref{sec:appl} is devoted to classifying central extensions
and studying the canonical class of a Courant-Dorfman algebra; in
Section \ref{sec:pb} we consider the non-degenerate case and derive
formulas for the Poisson bracket; here we also elucidate the
relation of our constructions with earlier work on Courant
algebroids. Finally, Section \ref{sec:concl} is devoted to
concluding remarks and speculations. For the convenience of the
reader we have also included several appendices where we have
collected the necessary facts about derivations, K\"{a}hler
differentials, Lie-Rinehart algebras and Leibniz algebras.

\begin{acknowledgement}
I thank Marius Crainic for urging me to work out explicit formulas
for the cochains in the standard complex of a Courant algebroid. The
results were presented at the workshop ``Aspects alg\'{e}briques et
g\'{e}om\'{e}trique des alg\`{e}bres de Lie" at Universit\'{e} de
Haute Alsace (Mulhouse, June 12-14, 2008), and at the Poisson 2008 conference at EPFL
(Lausanne, July 7-11, 2008); I thank the organizers of both meetings for this
opportunity. The paper was written up at Universiteit Utrecht and
the Max Planck Institut f\"{u}r Mathematik in Bonn, and I thank both
institutions for providing excellent working conditions and friendly
environment. Last but not least, I thank Yvette Kosmann-Schwarzbach and the anonymous referee
for their detailed and helpful comments and suggestions for improving the manuscript.
\end{acknowledgement}

\section{Definition and basic properties}\label{sec:def}

\subsection{Conventions and notation}
We fix once and for all a commutative ring $\KK$, containing $\hf$, as our ground
ring (the condition ensures that $\KK$-linear derivations annihilate
constants and polarization identities hold). All tensor products and
$\Hom$'s are assumed to be over $\KK$; tensor products and $\Hom$'s
over other rings will be explicitly indicated by appropriate
subscripts.

By a graded module we shall always mean a collection
$\M=\{\M_i\}_{i\in\ZZ}$ of modules indexed by $\ZZ$. The dual module
$\M^\vee$ is defined by setting $\M^\vee_i=(\M_{-i})^\vee$. For a
$k\in\ZZ$, the shifted module $\M[k]$ is defined by
$\M[k]_i=\M_{k+i}$, so that $(\M[k])^\vee=(\M^\vee)[-k]$.

The (graded) commutator of operators will always be denoted by
$\pbr$.

\subsection{Courant-Dorfman algebras and related categories}

\begin{defn}\label{def:DA}
A \emph{Courant-Dorfman algebra} consists of the following data:
\begin{itemize}
\item a commutative $\KK$-algebra $\R$;
\item an $\R$-module $\E$;
\item a symmetric bilinear form (\emph{pseudometric}) $\metr:\E\otimes_\R\E\To\R$;
\item a derivation $\dl:\R\To\E$;
\item a \emph{Dorfman bracket} $\brac:\E\otimes\E\To\E$.
\end{itemize}
These data are required to satisfy the following conditions:
\begin{enumerate}
\item $[e_1,fe_2]=f[e_1,e_2]+\met{e_1}{\dl f} e_2$;
\item $\met{e_1}{\dl\met{e_2}{e_3}}=\met{[e_1,e_2]}{e_3}+\met{e_2}{[e_1,e_3]}$;
\item $[e_1,e_2]+[e_2,e_1]=\dl\met{e_1}{e_2}$;
\item $[e_1,[e_2,e_3]]=[[e_1,e_2],e_3]+[e_2,[e_1,e_3]]$;
\item $[\dl f,e]=0$;
\item $\met{\dl f}{\dl g}=0$
\end{enumerate}
for all $e,e_1,e_2,e_3\in\E$, $f,g\in\R$.

When only conditions (1), (2) and (3) are satisfied, we shall speak
of an \emph{almost Courant-Dorfman algebra} and treat (4), (5) and
(6) as integrability conditions.
\end{defn}

\begin{rem}
A $\KK$-module $\E$ equipped with a bracket $\brac$ satisfying
condition (4) above is called a ($\KK$-) \emph{Leibniz algebra}. For
basic facts about these algebras we refer to Appendix
\ref{app:Leibniz}.
\end{rem}

Given a Courant-Dorfman algebra, the \emph{Courant bracket} $\cbr$
is defined by the formula
$$\cb{e_1}{e_2}=\hf([e_1,e_2]-[e_2,e_1])$$

Conversely, the Dorfman bracket can be recovered from the Courant
bracket:
$$[e_1,e_2]=\cb{e_1}{e_2}+\hf\dl\met{e_1}{e_2}$$

If $\thd\in\KK$, the definition of a Courant-Dorfman algebra can be
rewritten in terms of the Courant bracket, as was done originally in
\cite{LWX1}.

\begin{defn}\label{def:nondeg}
The bilinear form $\metr$ gives rise to a map
$$(\cdot)^\flat:\E\To\E^\vee=\Hom_\R(\E,\R)$$
defined by
$$e^\flat(e')=\met{e}{e'}$$
We say $\metr$ is \emph{strongly non-degenerate} if $(\cdot)^\flat$
is an isomorphism, and call a Courant-Dorfman algebra
\emph{non-degenerate} if its bilinear form is strongly non-degenerate. In
this case the inverse map is denoted by
$$(\cdot)^\sharp:\E^\vee\To\E$$
and there is a symmetric bilinear form
$$\pbr:\E^\vee\otimes_\R\E^\vee\To\R$$
defined by
\begin{equation}\label{eqn:inversemetr}
\pb{\lambda}{\mu}=\met{\lambda^\sharp}{\mu^\sharp}
\end{equation}
for $\lambda,\mu\in\E^\vee$.
\end{defn}

\begin{rem}
For non-degenerate Courant-Dorfman algebras, it can be shown that
conditions (1), (5) and (6) of Definition \ref{def:DA} are redundant.
\end{rem}

\begin{defn}
A \emph{strict morphism} between Courant-Dorfman algebras $\E$ and
$\E'$ is a map of $\R$-modules $f:\E\to\E'$ respecting all the
operations.
\end{defn}

\begin{rem}
It is possible to define a morphism of Courant-Dorfman algebras over
different base rings, as well as weak morphisms which preserve the
operations up to coherent homotopies. For the purposes of this
paper, strict morphisms over a fixed base suffice; we shall refer to
them simply as morphisms from now on.
\end{rem}

Courant-Dorfman algebras over a fixed $\R$ form a category, which we
denote by $\mathbf{CD}_\R$.

\begin{rem}
The Courant-Dorfman structure consists of several layers of
underlying structure: the $\R$-module $\E$, the \emph{metric
$\R$-module} $(\E,\metr)$, the \emph{differential metric
$\R$-module} $(\E,\metr,\dl)$ and the $\KK$-Leibniz algebra
$(\E,\brac)$. Correspondingly, there are obvious forgetful functors
from $\mathbf{CD}_\R$ to the categories $\mathbf{Mod}_\R$,
$\mathbf{Met}_\R$, $\mathbf{dMet}_\R$ and $\mathbf{Leib}_\KK$. We
shall refer to the respective fiber categories $\mathbf{CD}_\E$,
$\mathbf{CD}_{(\E,\metr)}$ and $\mathbf{CD}_{(\E,\metr,\dl)}$ when
we wish to consider Courant-Dorfman algebra with the indicated
underlying structure fixed. We shall frequently speak of just a
Courant bracket or a Dorfman bracket, with the rest of the data
implicitly understood.
\end{rem}

\begin{defn}\label{def:CAlgd}
Given a locally ringed space $(X,\O_X)$ over $\KK$, a \emph{Courant
algebroid} over $X$ is an $\O_X$-module $\E$ equipped with a
compatible Courant-Dorfman algebra structure.
\end{defn}

\begin{rem}
Definition \ref{def:CAlgd} differs somewhat from the earlier
versions. Traditionally \cite{LWX1}, $X$ was required to be a
$C^\infty$ manifold, $\E$ locally free of finite rank (i.e. sections
of a vector bundle), and $\metr$ strongly non-degenerate; Bressler
\cite{Bres} drops the finite-rank and non-degeneracy assumptions
while still requiring that $X$ be a smooth manifold. Our definition
is equivalent to those of loc. cit. under the aforementioned
additional assumptions.
\end{rem}

\subsection{The anchor, coanchor and tangent complex.}
Let $\Omega^1=\Omega^1_\R$ be the $\R$-module of K\"{a}hler
differentials, with the universal derivation
$$d_0:\R\To\Omega^1.$$
Furthermore, let
$$
\X^1=\X^1_\R=\mathrm{Der}(\R,\R)\simeq\Hom_{\R}(\Omega^1,\R)
$$
Now, let $(\R,\E)$ be a Courant-Dorfman algebra. By the universal
property of $\Omega^1$, there is a unique map of $\R$-modules
$$\delta:\Omega^1\To\E$$
such that $\delta(d_0f)=\dl f$ (see Appendix \ref{app:kaehler}).
This map will be referred to as the \emph{coanchor}. Define further
the \emph{anchor map}
$$\rho:\E\To\X^1$$
by setting
\begin{equation}\label{eqn:anchor}
\rho(e)\cdot f=\met{e}{\dl f}
\end{equation}
for all $e\in\E$, $f\in\R$.

\begin{rem}
In a non-degenerate almost Courant-Dorfman algebra, $\dl$ can be
recovered from $\rho$, and condition (1) of Definition \ref{def:DA}
follows from (2) and (3).
\end{rem}

The condition (6) of Definition \ref{def:DA} can now be restated as
\begin{equation}\label{eqn:cx}
\rho\circ\delta=0
\end{equation}
In other words, the following is a cochain complex of $\R$-modules:
\begin{equation}\label{eqn:tancx}
\Omega^1[2]\stackrel{\delta}{\To}\E[1]\stackrel{\rho}{\To}\X^1
\end{equation}
This complex will be denoted by $\T=\T_\E$ and referred to as the
\emph{tangent complex} of the Courant-Dorfman algebra $\E$; the
differential on $\T$ will also be denoted by $\delta$ (that is,
$\delta_{-2}=\delta$, $\delta_{-1}=\rho$).

\begin{defn}
A Courant-Dorfman algebra is \emph{exact} if its tangent complex is
acyclic.
\end{defn}

The complex $\T$ has an extra structure: namely, the symmetric
bilinear form $\metr$ on $\E$ extends to a graded
\emph{skew}-symmetric bilinear map of graded $\R$-modules
$$\Xi:\T\otimes_\R\T\To\R[2]$$
if we define
\begin{equation}\label{eqn:Xi}
\Xi(v,\alpha)=\iota_v\alpha=-\Xi(\alpha,v)
\end{equation}
$$
\Xi(e_1,e_2)=\met{e_1}{e_2}
$$
for $v\in\X^1$, $\alpha\in\Omega^1$, $e_1,e_2\in\E$.

\begin{prop}
$\Xi$ is $\delta$-invariant, i.e.
\begin{equation}\label{eqn:Xi-delta-inv}
\Xi(\delta a,b)+(-1)^{\mathrm{deg}(a)}\Xi(a,\delta b)=0
\end{equation}
for all homogeneous $a,b\in\T$.
\end{prop}

\begin{proof}
This amounts to saying that, for all $e\in\E$ and
$\alpha\in\Omega^1$, one has
\begin{equation}\label{eqn:rho-delta}
\met{e}{\delta\alpha}=\iota_{\rho(e)}\alpha,
\end{equation}
which is just a restatement of the definitions.
\end{proof}

\begin{prop}\label{prop:anc-is-hom}
The anchor $\rho$ is a homomorphism of Leibniz algebras.
\end{prop}

\begin{proof}
First, observe that conditions (3) and (5) of Definition \ref{def:DA}
imply
\begin{equation}\label{eqn:ax5'}
[e,\dl f]=\dl\met{e}{\dl f}
\end{equation}
Furthermore, by (2),
$$\met{e_1}{\dl\met{\dl f}{e_2}}=\met{[e_1,\dl f]}{e_2}+\met{\dl f}{[e_1,e_2]}$$
Combining these and using the definition of $\rho$, we immediately
get
$$\rho([e_1,e_2])\cdot f=\rho(e_1)\cdot(\rho(e_2)\cdot f)-\rho(e_2)\cdot(\rho(e_1)\cdot f),$$
as claimed.
\end{proof}

\begin{cor}
Let $(\R,\E)$ be a Courant-Dorfman algebra, and let $\K=\ker{\rho}$.
Then $(\R,\K)$ is a Courant-Dorfman subalgebra (with zero anchor).
\end{cor}

\begin{proof}
By Proposition \ref{prop:anc-is-hom}, $\K$ is closed under $\brac$;
by \eqref{eqn:cx}, the image of $\dl$ is contained in $\K$.
\end{proof}

\begin{prop}\label{prop:coanc-is-ideal}
The image $\delta\Omega^1$ is a two-sided ideal with respect to the
Dorfman bracket $\brac$. More precisely, the following identities
hold:
\begin{eqnarray*}
[e,\delta\alpha]&=&\delta L_{\rho(e)}\alpha \\
{[\delta\alpha,e]}&=&\delta(-\iota_{\rho(e)}d_0\alpha)
\end{eqnarray*}
In particular,
\begin{eqnarray*}
[\delta\alpha,\delta\beta]&=&0
\end{eqnarray*}
for all $\alpha,\beta\in\Omega^1$, $e\in\E$.
\end{prop}

\begin{proof}
For the first identity, it suffices to consider $\alpha$ of the form
$fd_0g$. The identity then follows by applying condition (1) and the
formula \eqref{eqn:ax5'}. The second identity then follows
immediately from condition (3) and the Cartan identity. The last
identity is then a consequence of \eqref{eqn:cx}.
\end{proof}

\begin{cor}\label{cor:Ebar}
Let $\bar\E=\E/\delta\Omega^1$. Then $(\R,\bar\E)$ is a Lie-Rinehart
algebra under the induced bracket and anchor; furthermore, the
pseudometric $\metr$ induces one on $\bar\K=\ker\bar\rho$ which is,
moreover, $\bar\E$-invariant (with respect to the natural action of
$\bar\E$ on $\bar\K$, see Appendix \ref{app:LR}).
\end{cor}

\begin{proof}
By Proposition \ref{prop:coanc-is-ideal}, the bracket on $\E$
descends to $\bar\E$; the induced bracket is skew-symmetric by
condition (3). Similarly, by \eqref{eqn:cx}, one gets the induced
anchor $\bar\rho:\bar\E\To\X^1$. The axioms for a Lie-Rinehart
algebra follow immediately from those for Courant-Dorfman algebra.

To prove the last statement, observe that
$\bar\K=\K/\delta\Omega^1$. Now, for all $e\in\K$,
$\alpha\in\Omega^1$,
\begin{equation}\label{eqn:K_ort_Omega}
\met{e}{\delta\alpha}=\iota_{\rho(e)}\alpha=0
\end{equation}
by \eqref{eqn:rho-delta}, hence $\metr$ descends to $\bar\K$. The
$\bar\E$-invariance follows from axiom (2). The equation
\eqref{eqn:K_ort_Omega} implies, in particular, that
$\delta\Omega^1$ is isotropic.
\end{proof}

\begin{rem}
Of course, $\E/\dl\R$ is always a Lie algebra (over $\KK$).
\end{rem}

\begin{defn}
Suppose $\E$ is a Courant-Dorfman algebra. An $\R$-submodule
$\D\subset\E$ is said to be a \emph{Dirac submodule} if $\D$ is
isotropic with respect to $\metr$ and is closed under $\brac$
(equivalently, under $\cbr$).
\end{defn}

\begin{prop}
If $\D$ is a Dirac submodule, $(\R,\D)$ is a Lie-Rinehart algebra
under the restriction of the anchor and bracket.
\end{prop}

\begin{proof}
Clear.
\end{proof}

Even though $\metr$ is allowed to be degenerate, even zero, it is
\emph{not} true that a Lie-Rinehart algebra is a special case of a
Courant-Dorfman algebra, because of the relation \eqref{eqn:anchor}
between the anchor and $\metr$. Nevertheless, the notion of a
morphism between a Courant-Dorfman algebra and a Lie-Rinehart
algebra does make sense.

\begin{defn}\label{def:morLR->CD}
A \emph{strict morphism} from a Lie-Rinehart algebra $\L$ to a
Courant-Dorfman algebra $\E$ is a map of $\R$-modules $p:\L\To\E$
satisfying the following conditions:
\begin{enumerate}
\item $p$ commutes with anchors and brackets;
\item $\metr\circ(p\otimes p)=0$
\end{enumerate}
\end{defn}

\begin{defn}\label{def:morCD->LR}
A \emph{strict morphism} from a Courant-Dorfman algebra $\E$ to a
Lie-Rinehart algebra $\L$ is an $\R$-module map $r:\E\To\L$
satisfying the following conditions:
\begin{enumerate}
\item $r$ commutes with anchors and brackets;
\item $r\circ\delta=0$
\end{enumerate}
\end{defn}

\begin{prop}\label{prop:morphisms}
The following are morphisms in the sense of the above definitions:
\begin{itemize}
\item the anchor $\rho:\E\To\X^1$;
\item the canonical projection $\pi:\E\To\bar{\E}$ from Corollary
\ref{cor:Ebar};
\item the inclusion $i:\D\To\E$ of a Dirac submodule.
\end{itemize}
\end{prop}
\begin{proof}
Obvious, in view of the already established facts.
\end{proof}

\subsection{Twists.} Given a Courant-Dorfman algebra $\E$ and a
3-form $\psi\in\Omega^3$, we can define a new bracket
\begin{equation}\label{eqn:twist}
[e_1,e_2]_\psi=[e_1,e_2]+\delta\iota_{\rho(e_2)}\iota_{\rho(e_1)}\psi
\end{equation}
This twisted bracket $\brac_\psi$ will be again a Dorfman bracket
(with the same $\metr$ and $\dl$) if and only if $d_0\psi=0$. It is
clear that this defines an invertible endofunctor
$\mathrm{Tw}(\psi)$ on the category $\mathbf{CD}_\R$, restricting to
each $\mathbf{CD}_{(\E,\metr,\dl)}$, and that
$$\mathrm{Tw}(\psi_1+\psi_2)=\mathrm{Tw}(\psi_1)\circ\mathrm{Tw}(\psi_2)$$
Furthermore, each $\beta\in\Omega^2$ defines a natural
transformation $\exp(-\beta)$ from $\mathrm{Tw}(\psi)$ to
$\mathrm{Tw}(\psi+d_0\beta)$ via
$$\exp(-\beta)(e)=e-\delta\iota_{\rho(e)}\beta$$
which is also additive. In fact, this yields an action of the group
crossed module $\Omega^2\stackrel{d_0}{\To}\Omega^{3,\mathrm{cl}}$
on the category $\mathbf{CD}_\R$, restricting to each
$\mathbf{CD}_{(\E,\metr,\dl)}$. In particular, the group
$\Omega^{2,\mathrm{cl}}$ of closed 2-forms acts on every Courant-Dorfman algebra by
automorphisms.

We refer to \cite{Bres} for the relevant calculations.

\subsection{Some examples.}\label{sec:examples}
\begin{eg}
Let $(\R,\E)$ be a Courant-Dorfman algebra with $\metr=0$. A quick
glance at the axioms then shows that $\E$ is a Lie algebra over
$\R$, while $\dl$ is a derivation with values in the center of $\E$.
There are no further restrictions.

As a special case of this, let $\E=\R$. Then the bracket must
vanish, while the derivation $\dl$ can be arbitrary.

More fundamentally, consider $\E=\Omega^1$ with $\dl=d_0$. This is
the initial object in $\mathbf{CD}_\R$.
\end{eg}

\begin{eg}
At the opposite extreme, let $\dl=0$. Then the definition reduces to
that of a quadratic Lie algebra over $\R$ (i.e. a Lie algebra
equipped with an ad-invariant quadratic form).
\end{eg}

\begin{eg}\label{eg:exact}
Given an $\R$, let $\Q_0=\X^1\oplus\Omega^1$. It becomes a
Courant-Dorfman algebra with respect to
\begin{eqnarray*}
\met{(v_1,\alpha_1)}{(v_2,\alpha_2)}&=&\iota_{v_1}\alpha_2+\iota_{v_2}\alpha_1\\
\dl f&=&(0,d_0f)\\
{[(v_1,\alpha_1),(v_2,\alpha_2)]}&=&(\pb{v_1}{v_2},L_{v_1}\alpha_2-\iota_{v_2}d_0\alpha_1)
\end{eqnarray*}

The bracket here is \emph{the} original Dorfman bracket
\cite{Dorf87}, while the corresponding Courant bracket is
$$
\cb{(v_1,\alpha_1)}{(v_2,\alpha_2)}=(\pb{v_1}{v_2},L_{v_1}\alpha_2-L_{v_2}\alpha_1-\hf
d_0(\iota_{v_1}\alpha_2-\iota_{v_2}\alpha_1))
$$
which is \emph{the} original Courant bracket \cite{Cou}.

For any $\psi\in\Omega^{3,\cl}$, the Courant-Dorfman algebra
$\Q_\psi=\Tw(\psi)(\Q_0)$ is exact. Conversely, it can be shown
\cite{Sev-Letters} that, if $\Q$ is exact and its tangent complex
$\T_\Q$ \eqref{eqn:tancx} admits an isotropic splitting, $\Q$ is
isomorphic to $\Q_\psi$ for some $\psi$; since isotropic splittings
form an $\Omega^2$-torsor, such exact Courant-Dorfman algebras are
classified by $H^3_{\mathrm{dR}}(\R)$.
\end{eg}

\begin{eg}\label{eg:bialg}
As a variant of the previous example, we can replace $\X^1$ by
an arbitrary Lie-Rinehart algebra $(\R,\L)$, and let
$\E=\L\oplus\Omega^1$. Given any $\psi\in\Omega^{3,\cl}$, define the
structure maps as follows:
\begin{eqnarray*}
\met{(a_1,\alpha_1)}{(a_2,\alpha_2)}&=&\iota_{\rho(a_1)}\alpha_2+\iota_{\rho(a_2)}\alpha_1\\
\dl f&=&(0,d_0f)\\
{[(a_1,\alpha_1),(a_2,\alpha_2)]}&=&([a_1,a_2],L_{\rho(a_1)}\alpha_2-\iota_{\rho(a_2)}d_0\alpha_1+
\iota_{\rho(a_1)}\iota_{\rho(a_2)}\psi),
\end{eqnarray*}
where $\rho$ is the anchor of $\L$.

More generally, we can consider a pair of compatible
Lie-Rinehart algebras in duality (a Lie bialgebroid) \cite{LWX1}.
\end{eg}

\begin{eg}\label{eg:Bloch}
Consider a Lie algebra $\g$ over $\KK$ equipped with an ad-invariant
pseudometric $\metr$. Given a $\KK$-algebra $\R$, let
$\ug=\R\otimes\g$; extend $\brac$ and $\metr$ to $\ug$ by
$\R$-linearity. Finally, let $\E=\ug\oplus\Omega^1$ and define the
structure maps as follows:
\begin{eqnarray*}
\met{(F_1,\alpha_1)}{(F_2,\alpha_2)}&=&\met{F_1}{F_2}\\
\dl f&=&(0,d_0f)\\
{[(F_1,\alpha_1),(F_2,\alpha_2)]}&=&([F_1,F_2],\met{d_0F_1}{F_2}),
\end{eqnarray*}
where, for $F_i=f_i\otimes x_i$ ($i=1,2$), $\met{d_0F_1}{F_2}$ means
$\met{x_1}{x_2}(d_0f_1)f_2$. Again, it can be easily checked that this
defines a Courant-Dorfman structure (with zero anchor map). This
algebra goes back to the work of Spencer Bloch on algebraic
$K$-theory (see \cite{Bres} and references therein).
\end{eg}

\begin{eg}
As a special case of the previous example, assume that
$\QQ\subset\KK$ and consider $\R=\KK[z,z^{-1}]$, the ring of Laurent
polynomials. In this case, $\ug$ is better known as $L\g$, the loop
Lie algebra of $\g$, and the Lie algebra structure on
$\E/\dl\R=L\g\oplus(\Omega^1/d_0\R)\simeq L\g\oplus\KK$ is very
well-known. The latter isomorphism is induced by the residue map
$$\oint:\Omega^1\To\KK,$$
and so the Lie bracket is given by the
famous Kac-Moody formula
$$[F_1,F_2]+\oint\met{d_0F_1}{F_2}$$

\end{eg}

\begin{eg}
It is possible to combine Examples \ref{eg:bialg} and
\ref{eg:Bloch}. Let $\ug$ be a Lie algebra over $\R$. Assume there
is a connection $\nabla$ on $\ug$ which acts by derivations of
the Lie bracket. Let $\omega\in\Omega^2\otimes_\R\ug$ be the curvature.
Then $\L=\X^1\oplus\ug$ becomes a Lie-Rinehart algebra with the
bracket given by
$$[(v_1,\xi_1),(v_2,\xi_2)]=([v_1,v_2],[\xi_1,\xi_2]+\nabla_{v_1}\xi_2-\nabla_{v_2}\xi_1+\iota_{v_2}\iota_{v_1}\omega)$$
and the anchor given by projection onto the first factor.

Suppose now that $\ug$ is equipped with an ad-invariant
$\nabla$-invariant (i.e. $\L$-invariant) pseudometric $\metr$ and
that, moreover, there exists a 3-form $\psi\in\Omega^3$ such that
$$d_0\psi=\hf\met{\omega}{\omega}$$
(the condition for (the one half of) the first Pontryagin class to vanish). Then the
Lie-Rinehart structure on $\L$ extends to a Courant-Dorfman
structure on $\E=\L\oplus\Omega^1$ as follows:
\begin{eqnarray*}
\dl f&=&(0,0,d_0f)\\
\met{(v_1,\xi_1,\alpha_1)}{(v_2,\xi_2,\alpha_2)}&=&\iota_{v_1}\alpha_2+\iota_{v_2}\alpha_1+\met{\xi_1}{\xi_2}\\
{[(v_1,\xi_1,\alpha_1),(v_2,\xi_2,\alpha_2)]}&=&([v_1,v_2],[\xi_1,\xi_2]+\nabla_{v_1}\xi_2-\nabla_{v_2}\xi_1+\iota_{v_2}\iota_{v_1}\omega,\\
&&\met{\nabla\xi_1}{\xi_2}+\met{\xi_1}{\iota_{v_2}\omega}-\met{\xi_2}{\iota_{v_1}\omega}+\\
&&+L_{v_1}\alpha_2-\iota_{v_2}d_0\alpha_1+\iota_{v_2}\iota_{v_1}\psi)
\end{eqnarray*}
We refer to \cite{Bres} for the relevant calculations.
\end{eg}

\section{A preliminary construction: universal enveloping and convolution
algebras.}\label{sec:conv}

Let $V$ and $W$ be $\KK$-modules, and let $\rbr:V\otimes V\to W$ be a symmetric
bilinear form. Consider the graded $\KK$-module $L=W[2]\oplus V[1]$;
it becomes a graded Lie algebra over $\KK$ with the only nontrivial
brackets given by $-\rbr$. Consider its universal enveloping algebra
$U(L)$. As an algebra, it is a quotient of the tensor algebra $T(L)$
(with grading induced by that of $L$) by the homogeneous ideal
generated by elements of the form $v_1\otimes v_2+v_2\otimes
v_1+(v_1,v_2)$, $v\otimes w+w\otimes v$ and $w_1\otimes
w_2+w_2\otimes w_1$. Consequently, for $p\geq 0$, we have
$$U(L)_{-p}=\bigoplus_{k=0}^{\left[\frac{p}{2}\right]}(V^{\otimes(p-2k)}\otimes S^kW)/R$$
where $R$ is the submodule generated by elements of the form
$$v_1\otimes\cdots\otimes v_i\otimes v_{i+1}\otimes\cdots\otimes v_{p-2k}\otimes w_1\cdots
w_k+
$$
$$
+v_1\otimes\cdots\otimes v_{i+1}\otimes v_i\otimes\cdots\otimes
v_{p-2k}\otimes w_1\cdots w_k+
$$
$$
+v_1\otimes\cdots\otimes\hat{v}_i\otimes\hat{v}_{i+1}\otimes\cdots\otimes
v_{p-2k}\otimes(v_i,v_{i+1})w_1\cdots w_k
$$
for $i=1,\ldots,p-2k-1$, $k=0,\ldots,\left[\frac{p}{2}\right]$.

Recall that $U(L)$ is also a graded cocommutative coalgebra with
comultiplication
$$\Delta:U(L)\To U(L)\otimes U(L)$$
uniquely determined by the requirement that the elements of $L$ be
primitive and that $\Delta$ be an algebra homomorphism. Explicitly,
$$\Delta(v_1\cdots v_{p-2k}w_1\cdots w_k)=$$
$$=\sum_{i=0}^k\sum_{j=0}^{p-2k}\sum_{\sigma,\tau}(-1)^\sigma v_{\sigma(1)}\cdots v_{\sigma(j)}w_{\tau(1)}\cdots w_{\tau(i)}
\otimes v_{\sigma(j+1)}\cdots v_{\sigma(p-2k-j)}w_{\tau(i+1)}\cdots
w_{\tau(k)}
$$
where $\sigma$ runs over $(j,p-2k-j)$-shuffles, and $\tau$ runs over
$(i,k-i)$-shuffles.

Now, recall that, whenever $U$ is a graded $\KK$-coalgebra and $\R$
is a graded $\KK$-algebra, the graded $\R$-module $\Hom(U,\R)$ is
naturally an $\R$-algebra, called the \emph{convolution algebra}
(this is a general fact about a pair (comonoid, monoid) in any
monoidal category).

Let us apply this construction to $U=U(L)$ and an arbitrary
$\KK$-algebra $\R$ (concentrated in degree 0). Denote the
corresponding convolution algebra by $\A=\A(V,W;\R)=\Hom(U(L),\R)$.
Since $U(L)$ is non-positively graded and $\R$ sits in degree 0,
$\A$ is non-negatively graded. Explicitly, for $p\geq 0$, $\A^p$
consists of ($\left[\frac{p}{2}\right]+1$)-tuples
$$\omega=(\omega_0,\omega_1,\ldots,\omega_{\left[\frac{p}{2}\right]})$$
where
$$\omega_k:V^{\otimes^{p-2k}}\otimes W^{\otimes^k}\To\R$$
is symmetric in the $W$-arguments and satisfying
\begin{equation}\label{eqn:weakantisym}
\omega_k(\ldots,v_i,v_{i+1},\ldots;\ldots)+\omega_k(\ldots,v_{i+1},v_i,\ldots;\ldots)=
\end{equation}
$$=-\omega_{k+1}(\ldots,\hat{v}_i,\hat{v}_{i+1},\ldots;(v_i,v_{i+1}),\ldots)$$
for all $i=1,\ldots,p-2k-1$. By adjunction, $\omega_k$ can be viewed
as a map

$$V^{\otimes^{p-2k}}\To\Hom(S^kW,\R)$$

Again, since $S(W[2])$ is a coalgebra (concentrated in even
non-positive degrees), $\Hom(S(W[2]),\R)$ is an algebra with
multiplication given by

\begin{equation}\label{eqn:product'}
HK(w_1,\ldots,w_{i+j})=
\end{equation}
$$\sum_{\tau\in\sh(i,j)}H(w_{\tau(1)},\ldots,w_{\tau(i)})K(w_{\tau(i+1)},\ldots,w_{\tau(i+j)})$$
This leads to the following formula for the multiplication in $\A$:

\begin{equation}\label{eqn:product}
(\omega\eta)_k(v_1,\ldots,v_{p+q-2k})=
\end{equation}
$$=\sum_{i+j=k}\sum_{\sigma\in\mathrm{sh}(p-2i,q-2j)}(-1)^\sigma\omega_i(v_{\sigma(1)},\ldots,v_{\sigma(p-2i)})
\eta_j(v_{\sigma(p-2i+1)},\ldots,v_{\sigma(p+q-2k)})
$$
where the multiplication in each summand takes place in
$\Hom(S(W[2]),\R)$ according to formula \eqref{eqn:product'}. In
particular,
\begin{equation}\label{eqn:product0}
(\omega\eta)_0(v_1,\ldots,v_{p+q})=
\end{equation}
$$=\sum_{\sigma\in\mathrm{sh}(p,q)}(-1)^\sigma\omega_0(v_{\sigma(1)},\ldots,v_{\sigma(p)})
\eta_0(v_{\sigma(p+1)},\ldots,v_{\sigma(p+q)})
$$
where the multiplication in each summand takes place in $\R$.

Recall further that, as any universal enveloping algebra, $U(L)$ has
a canonical increasing filtration
$$\KK=U^0\subset U^1\subset\cdots\subset U^{n}\subset U^{n+1}\subset\cdots\subset U(L)$$
where $U^n$ is the submodule spanned by products of no more than $n$
elements of $L$. This induces a decreasing filtration of
$\A=\A(V,W;\R)$ by $\R$-submodules
$$\A\supset\mathrm{Ann}(U^1)\supset\cdots\supset\mathrm{Ann}(U^n)\supset\mathrm{Ann}(U^{n+1})\supset\cdots\supset 0$$
where $\mathrm{Ann}(U^n)$ denotes the annihilator of $U^n$ in $\A$.
Now, given $q,i\geq 0$, define
$$\A^q_i=\mathrm{Ann}(U^{q-i-1})\cap\A^q$$
and
$$\A_i=\bigoplus_{q\geq 0}\A^q_i$$
(set $\A_i=0$ for $i<0$). It is easy to see that this defines an
\emph{increasing} filtration on $\A$ which is finite in each
(superscript) degree, and that, furthermore,
$\A_i\A_j\subset\A_{i+j}$ (in particular, $\A_0$ is a subalgebra of
$\A$ with respect to the multiplication \eqref{eqn:product0}).
Explicitly,
$$\A_i=\{\omega\in\A|\omega_k=0,\forall k>i\}.$$
Define, as usual, $\gr_i\A^q:=\A^q_i/\A^q_{i-1}$, and
let
$$\gr\A^q=\bigoplus_i\gr_i\A^q$$
The following is then immediate:
\begin{prop}
There is a canonical isomorphism of graded $\R$-modules
$$\gr\A\simeq\Hom(S(L),\R)$$
where the grading on the left hand side is with respect to the
superscript degree. In particular,
$$\A_0=\gr_0\A=\Hom(S(V[1]),\R)$$
\end{prop}

\begin{rem}
If $\KK\supset\QQ$, the symmetrization map
$$\Psi:S(L)\To U(L)$$
is a coalgebra isomorphism by the Poincar{\'e}-Birkhoff-Witt theorem.
Hence, for any $\R$, the dual map
$$\Psi^*:\Hom(U(L),\R)\To\Hom(S(L),\R)=\Hom(S(V[1]),\Hom(S(W[2]),\R))$$
is an isomorphism of algebras. Explicitly,
\begin{equation}\label{eqn:symmetrization}
(\Psi^*\omega)_k(v_1,\ldots,v_{p-2k})=\frac{1}{(p-2k)!}\sum_{\sigma\in
S_{p-2k}}(-1)^\sigma \omega_k(v_{\sigma(1)},\ldots,v_{\sigma(p-2k)})
\end{equation}
\end{rem}

\section{The standard complex.}\label{sec:complex}

\subsection{The algebra $\C(\E,\R)$}
Let $(\E,\metr)$ be a metric $\R$-module; consider the convolution
algebra $\A=\A(\E,\Omega^1;\R)$ as in the previous section, with
$\rbr=d_0\metr$. Let $\C^0=\R$ and for each $p> 0$, define the
submodule $\C^p\subset\A^p$ as consisting of those
$\omega=(\omega_0,\omega_1,\ldots)$ which satisfy the following two
additional conditions:

\begin{enumerate}
\item Each $\omega_k$ takes values in
$\X^k=\Hom_\R(S_\R^k\Omega^1,\R)\subset\Hom(S^k\Omega^1,\R)$;
\item Each $\omega_k$ is $\R$-linear in the last (($p-2k$)-th) argument.
\end{enumerate}

For $e_1,\ldots,e_{p-2k}\in\E$, $\omega_k(e_1,\ldots,e_{p-2k})$ can
be viewed as either a symmetric $k$-derivation of $\R$ whose value
on $f_1,\ldots,f_k\in\R$ will be denoted by
$$\omega_k(e_1,\ldots,e_{p-2k};f_1,\ldots,f_k)$$
or as a symmetric $\R$-multilinear function on $\Omega^1$ whose
value on a $k$-tuple $\alpha_1,\ldots,\alpha_k$ will be similarly
denoted by
$$\bar{\omega}_k(e_1,\ldots,e_{p-2k};\alpha_1,\ldots,\alpha_k)$$
so that
$$\bar{\omega}_k(\ldots;d_0f_1,\ldots,d_0f_k)=\omega_k(\ldots;f_1,\ldots,f_k)$$
Evidently $\bar{\omega}_0=\omega_0$. Often we shall drop the bar
from the notation altogether when it is not likely to cause
confusion.

\begin{rem}\label{rem:full}
If $\metr$ is \emph{full}, in the sense that there exist $e_i,e_i'\in\E$, $i=1,\ldots,N$, such that $\sum_i\met{e_i}{e_i'}=1$, then an
$\omega=(\omega_0,\omega_1,\ldots)$ is uniquely determined by
$\omega_0$, for then
$$\omega_1(e_1,\ldots;f)=-\sum_i(\omega_0(fe_i,e_i',e_1,\ldots)+\omega_0(e_i',fe_i,e_1,\ldots))$$
and so on by induction. This condition is very often satisfied and
is a great help when one needs to prove, for instance, that some
cochain vanishes.
\end{rem}

\begin{prop}\label{prop:1stOrd}
For all $1\leq i<p-2k$ the following holds:
$$\bar{\omega}_k(\ldots,fe_i,\ldots)=f\bar{\omega}_k(\ldots,e_i,\ldots)+$$
$$+\sum_{j=1}^{p-2k-i}(-1)^j\met{e_i}{e_{i+j}}\iota_{d_0f}\bar{\omega}_{k+1}(\ldots,\hat{e}_i,\ldots,\hat{e}_{i+j},\ldots)$$
where $\iota_\alpha$ denotes contraction with $\alpha\in\Omega^1$.
\end{prop}
\begin{proof}
By induction from $i=p-2k$ downward, using \eqref{eqn:weakantisym}
at each step.
\end{proof}

Define $\C=\C(\E,\R)=\{\C^p\}_{p\geq 0}$.

\begin{prop}
$\C(\E,\R)\subset\A(\E,\Omega^1;\R)$ is a graded subalgebra.
\end{prop}
\begin{proof}
Given $\omega\in\C^p$, $\eta\in\C^q$ we must show that $\omega\eta$
satisfies conditions (1) and (2) defining $\C$. The first one is
clear, while the second one follows from the observation that, since
the expression \eqref{eqn:product} for $(\omega\eta)_k$ is a sum
over \emph{shuffle} permutations, the last argument of
$(\omega\eta)_k$ occurs either as the last argument of $\omega_i$ or
the last argument of $\eta_j$.
\end{proof}

Let $s:(\E,\metr)\To(\E',\metr')$ be a map of metric $\R$-modules.
It induces a map $s^\vee:\C(\E',\R)\To\C(\E,\R)$ given by
$$(s^\vee\omega)_k(e_1,\ldots,e_{q-2k})=\omega_k(s(e_1),\ldots,s(e_{q-2k})),$$
for every $\omega\in\C^q(\E',\R)$.
This map is obviously a morphism of graded $\R$-algebras. In other
words,

\begin{prop}\label{prop:Met->gra}
The assignment $(\E,\metr)\mapsto\C(\E,\R)$, $s\mapsto s^\vee$ is a
contravariant functor from the category $\mathbf{Met}_\R$ of metric
$\R$-modules to the category $\mathbf{gra}_\R$ of graded commutative
$\R$-algebras.
\end{prop}

\subsection{The filtration $\{\C_i\}_{i\geq 0}$}

The filtration $\{\A_i\}$ on $\A$ induces one on $\C$ by
$\C_i=\A_i\cap\C$.

\begin{prop}
There is a canonical isomorphism of graded $\R$-modules
$$\gr\C\simeq\Hom_\R(S_\R(\E[1]\oplus\Omega^1[2]),\R)$$
In particular,
$$\C_0=\Hom_\R(S_\R(\E[1]),\R)$$
is a subalgebra of $\C$.
\end{prop}
\begin{proof}
Observe that, if
$\omega=(\omega_0,\ldots,\omega_i,0,\ldots)\in\C^q_i$, then
$\omega_i$ is completely skew-symmetric in the first $q-2i$
variables and hence $\R$-linear in each of them by Proposition
\ref{prop:1stOrd}. Clearly, $\omega_i$ only depends on the class of
$\omega$ in $\gr_i\C^q$, and vanishes if and only if
$\omega\in\C_{i-1}^q$.
\end{proof}

\begin{rem}\label{rem:fingen}
Observe that, in particular, $\C^0=\C^0_0=\R$ and
$\C^1=\C^1_0=\Hom_\R(\E,\R)=\E^\vee$. One always has the natural
inclusion $\Lambda_\R\E^\vee\hookrightarrow\C_0$. If $\E$ is
sufficiently nice (e.g. locally free of finite rank), this inclusion
is an isomorphism, so that $\C_0$ is generated as an algebra by
$\C^{\leq 1}$. Moreover, in that case $\C^{\leq 2}$ generates
\emph{all} of $\C$.
\end{rem}

\begin{rem}
If $\KK\supset\QQ$, the image of $\C(\E,\R)$ under the
symmetrization map $\Psi^*$ \eqref{eqn:symmetrization} is the
subalgebra $\hat{\C}(\E,\R)$ of
$\Hom(S(\E[1]),\Hom(S(\Omega^1[2]),\R)$ consisting of those
$\hat{\omega}=(\hat{\omega}_0,\hat{\omega}_1,\ldots)$ which satisfy
the following two conditions:
\begin{enumerate}
\item Each $\hat{\omega}_k$ takes values in
$\X^k=\Hom_\R(S_\R^k\Omega^1,\R)\subset\Hom(S^k\Omega^1,\R)$;
\item For any $i=1,\ldots,\deg{\hat{\omega}}-2k$ and $f\in\R$,
$$\hat{\omega}_k(\ldots,fe_i,\ldots)=f\hat{\omega}_k(\ldots,e_i,\ldots)+
$$
$$
+\hf\sum_{j\neq
i}(-1)^{i-j+\theta(i-j)}\met{e_i}{e_{j}}\iota_{d_0f}\hat{\omega}_{k+1}(\ldots,\hat{e}_i,\ldots,\hat{e}_j,\ldots)$$
where $\theta$ is the Heaviside function (so that
$(-1)^{\theta(i-j)}=\frac{j-i}{|j-i|}$).
\end{enumerate}
This algebra is relevant for the Courant bracket-based formulation,
which some researchers may prefer.
\end{rem}

\subsection{The differential.} Suppose now that $(\R,\E)$ is
equipped with an almost Courant-Dorfman structure. For
$\eta\in\C^q(\E,\R)$, define $d\eta=((d\eta)_0,(d\eta)_1,\ldots)$ by
setting
\begin{equation}\label{eqn:diff}
(d\eta)_k(e_1,\ldots,e_{q-2k+1};f_1,\ldots,f_k)=
\end{equation}
$$=\sum_{\mu=1}^k\eta_{k-1}(\dl f_\mu,e_1,\ldots,e_{q-2k+1};f_1,\ldots,\widehat{f_\mu},\ldots,f_k)+$$
$$+\sum_{i=1}^{q-2k+1}(-1)^{i-1}\met{e_i}{\dl(\eta_k(e_1,\ldots,\widehat{e_i},\ldots,e_{q-2k+1};f_1,\ldots,f_k))}+$$
$$+\sum_{i<j}(-1)^i\eta_k(e_1,\ldots,\widehat{e_i},\ldots,\widehat{e_j},[e_i,e_j],e_{j+1},\ldots,e_{q-2k+1};f_1,\ldots,f_k)$$

\begin{prop} The operator $d$ is a derivation of the algebra $\C(\E,\R)$
of degree +1; if the almost Courant-Dorfman structure is a
Courant-Dorfman structure, it squares to zero.
\end{prop}

In this generality, the only proof we have is a verification of all
the claims (that $d\C^q\subset\C^{q+1}$, $d$ is a derivation and
$d^2=0$) by a direct calculation. It is completely straightforward
but extremely tedious; to save space and time, we omit it. However,
it is worth noting that, under the conditions of Remark
\ref{rem:fingen}, it suffices to do the calculations in low degrees.
We display these calculations here as it is certainly instructive to
see how the conditions (4), (5) and (6) of Def \ref{def:DA} imply
that $d^2=0$. Thus, for $f\in\C^0=\R$ we have
$df=(df)_0\in\C^1=\E^\vee$ with
$$(df)_0(e)=\met{e}{\dl f}=\rho(e)f$$
whereas for $\lambda\in\C^1$ we have
$d\lambda=((d\lambda)_0,(d\lambda)_1)$ with
\begin{eqnarray*}
(d\lambda)_0(e_1,e_2)&=&\rho(e_1)(\lambda(e_2))-\rho(e_2)(\lambda(e_1))-\lambda([e_1,e_2])\\
(d\lambda)_1(g)&=&\lambda(\dl g)
\end{eqnarray*}
Therefore,
$$(d(df))_0(e_1,e_2)=\rho(e_1)(\rho(e_2)f)-\rho(e_2)(\rho(e_1)f)-\rho([e_1,e_2])f=0$$
by Proposition \ref{prop:anc-is-hom}, while
$$(d(df))_1(g)=df(\dl g)=\met{\dl g}{\dl f}=0$$
by condition (6) of Definition \ref{def:DA}.

Now, if $\omega=(\omega_0,\omega_1)\in\C^2$,
$d\omega=((d\omega)_0,(d\omega)_1)$ where
\begin{eqnarray*}
(d\omega)_0(e_1,e_2,e_3)&=&\rho(e_1)\omega_0(e_2,e_3)-\rho(e_2)\omega_0(e_1,e_3)+\rho(e_3)\omega_0(e_1,e_2)-\\
&&-\omega_0([e_1,e_2],e_3)-\omega_0(e_2,[e_1,e_3])+\omega_0(e_1,[e_2,e_3])\\
(d\omega)_1(e,f)&=&\omega_0(\dl f,e)+\rho(e)\omega_1(f)
\end{eqnarray*}
from which we obtain, using Proposition \ref{prop:anc-is-hom}:
$$(d(d\lambda))_0(e_1,e_2,e_3)=\lambda([[e_1,e_2],e_3]+[e_2,[e_1,e_3]]-[e_1,[e_2,e_3]])=0$$
by condition (4) of Def \ref{def:DA}, and
$$(d(d\lambda))_1(e,f)=\rho(\dl f)\lambda(e)-\lambda([\dl f,e])=0$$
by conditions (6) and (5) of Definition \ref{def:DA}.

\begin{cor}
Given $\eta\in\C^q$,
$\overline{d\eta}=((\overline{d\eta})_0,(\overline{d\eta})_1,\ldots)$
is given by
\begin{equation}\label{eqn:diff'}
(\overline{d\eta})_k(e_1,\ldots,e_{q-2k+1};\alpha_1,\ldots,\alpha_k)=
\end{equation}
$$
=\sum_{\mu=1}^k\bar{\eta}_{k-1}(\delta\alpha_\mu,e_1,\ldots;\alpha_1,\ldots,\widehat{\alpha_\mu},\ldots,\alpha_k)+
$$
$$
+\sum_{i=1}^{q-2k+1}(-1)^{i-1}\rho(e_i)\bar{\eta}_k(e_1,\ldots,\widehat{e_i},\ldots;\alpha_1,\ldots,\alpha_k)+
$$
$$
+\sum_{i=1}^{q-2k+1}\sum_{\mu=1}^k(-1)^i\bar{\eta}_k(e_1,\ldots,\widehat{e_i},\ldots;\alpha_1,\ldots,\iota_{\rho(e_i)}d_0\alpha_\mu,\ldots,\alpha_k)+
$$
$$
+\sum_{i<j}(-1)^i\bar{\eta}_k(e_1,\ldots,\widehat{e_i},\ldots,\widehat{e_j},[e_i,e_j],e_{j+1},\ldots;\alpha_1,\ldots,\alpha_k)
$$
\end{cor}
\begin{proof}
It is obvious that
$$(\overline{d\eta})_k(\ldots;d_0f_1,\ldots,d_0f_k)=(d\eta)_k(\ldots;f_1,\ldots,f_k)$$
so we only have to prove $\R$-linearity in the $\alpha$'s. This is
done with the help of Proposition \ref{prop:1stOrd}.
\end{proof}

If $s:\E\To\E'$ is a strict morphism of Courant-Dorfman algebras, a
quick inspection of the formulas reveals that $s^\vee$ commutes with
differentials. We summarize the preceding discussion in our main
theorem, extending Proposition \ref{prop:Met->gra}:

\begin{thm}\label{thm:CD->DGA}
The assignment $(\R,\E)\mapsto(\C(\E,\R),d)$, $s\mapsto s^\vee$ is a
contravariant functor from the category $\mathbf{CD}_\R$ of
Courant-Dorfman algebras over $\R$ and strict morphisms to the
category $\mathbf{dga}_\R$ of differential graded algebras with
zero-degree component equal to $\R$ and $\R$-linear dg morphisms.
\end{thm}

\begin{rem}
The tangent complex $\T_\E$ we have constructed \eqref{eqn:tancx} is
indeed the tangent complex of the dg algebra $\C(\E,\R)$ in the
sense of algebraic geometry.
\end{rem}

\begin{cor}
Given a locally ringed space $(X,\O_X)$, there is a (covariant)
functor from the category $\mathbf{CA}_X$ of Courant algebroids over
$X$ to the category $\mathbf{dgS}_X$ of differential graded spaces
over $X$.
\end{cor}

The complex $(\C(\E,\R),d)$ will be referred to as the
\emph{standard complex} of $(\R,\E)$, and its $q$-th cohomology
module will be denoted by $H^q(\E,\R)$. It is an analogue, for
Courant-Dorfman algebras, of the de Rham complex of a Lie-Rinehart
algebra $(\R,\L)$ (see Appendix \ref{app:LR}). In the latter case
there is an evident chain map from the de Rham complex to the
Chevalley-Eilenberg complex of the Lie algebra $\L$ with
coefficients in the module $\R$. There is an analogous statement in
our case:

\begin{prop}
The assignment $\eta\mapsto\eta_0$ is a chain map from the standard
complex $\C(\E,\R)$ to the Loday-Pirashvili complex
$\C_{\mathrm{LP}}(\E,\R)$ of the Leibniz algebra $\E$ with
coefficients in the symmetric $\E$-module $\R$.
\end{prop}

\begin{proof}
We have
\begin{equation}\label{eqn:diff0}
(d\eta)_0(e_1,\ldots,e_{q+1})=
\end{equation}
$$=\sum_{i=1}^{q+1}(-1)^{i-1}\rho(e_i)\eta_0(e_1,\ldots,\widehat{e_i},\ldots,e_{q+1})+$$
$$+\sum_{i<j}(-1)^i\eta_0(e_1,\ldots,\widehat{e_i},\ldots,\widehat{e_j},[e_i,e_j],e_{j+1},\ldots,e_{q+1})$$
which coincides with the expression \eqref{eqn:d_LP_sym} for
$d_{LP}(\eta_0)$, where one defines
$$
[e,f]=\rho(e)f=-[f,e].
$$
\end{proof}

\subsection{On morphisms between Lie-Rinehart and Courant-Dorfman
algebras}

Let $\L$ be a Lie-Rinehart algebra and $\E$ a Courant-Dorfman algebra.
Suppose $p:\L\To\E$ is a morphism in the sense of Definition
\ref{def:morLR->CD}. We define the induced map
$$p^\vee:\C(\E;\R)\To\widetilde{\Omega}(\L,\R)$$
(see Appendix \ref{app:LR}) by setting
$$(p^\vee\omega)(x_1,\ldots,x_q)=\omega_0(p(x_1),\ldots,p(x_q))$$
Similarly, given a morphism $r:\E\To\L$ in the sense of Definition
\ref{def:morCD->LR}, we define
$$r^\vee:\widetilde{\Omega}(\L,\R)\To\C(\E,\R)$$
by
\begin{eqnarray*}
(r^\vee\omega)_0(e_1,\ldots,e_q)&=&\omega(r(e_1),\ldots,r(e_q))\\
(r^\vee\omega)_{i>0}&=&0
\end{eqnarray*}
(so the image of $r^\vee$ is contained in $\C_0(\E,\R)$).

\begin{prop}
The maps $p^\vee$ and $r^\vee$ are morphisms of differential graded
algebras.
\end{prop}
\begin{proof}
Given $\omega\in\C^q(\L,\R)$, $(dr^\vee\omega)_0=(r^\vee d\omega)_0$
by condition (1) of Definition \ref{def:morCD->LR} and because for
alternating cochains, the Loday-Pirashvili formula \eqref{eqn:diff0}
coincides with the Cartan-Chevalley-Eilenberg formula
\eqref{eqn:d_LR}, whereas $(dr^\vee\omega)_1=0$ by condition (2) of
Definition \ref{def:morCD->LR}.

On the other hand, $dp^\vee-p^\vee d=0$ by formula
\eqref{eqn:weakantisym} and conditions (1) and (2) of Definition
\ref{def:morLR->CD}. Details are left to the reader.
\end{proof}

In particular, the morphisms $\rho:\E\To\X^1$, $\pi:\E\To\bar{\E}$
and $i:\D\To\E$ (see Prop. \ref{prop:morphisms}) give rise to the
corresponding dg maps $\rho^\vee:\widetilde{\Omega}_\R\To\C(\E,\R)$,
$\pi^\vee:\widetilde{\Omega}(\bar{\E},\R)\To\C(\E,\R)$ and
$i^\vee:\C(\E,\R)\To\widetilde{\Omega}(\D,\R)$.
\\
\\
Finally, since the de Rham algebra $(\Omega_\R,d_0)$ is initial in
the category $\mathbf{dga}_\R$, there is an evident dg map from
$\Omega_\R$ to each of these dg algebras, commuting with the above
maps. We shall denote this map by $\rho^*$, where $\rho$ is the
anchor. Explicitly,
\begin{eqnarray}\label{eqn:rhostar1}
(\rho^*\omega)_0(e_1,\ldots,e_q)&=&\iota_{\rho(e_q)}\cdots\iota_{\rho(e_1)}\omega\\
(\rho^*\omega)_{>0}&=&0 \label{eqn:rhostar2}
\end{eqnarray}

\subsection{Filtration $\{\F_i\}_{i\geq 0}$ and ideal
$\I$}\label{subsec:filtrF}
Observe that the differential $d$ does not preserve the filtration
$\{\C_i\}$. In fact, for $\omega\in\C_k$, $\omega_{k+1}=0$ but
$$(d\omega)_{k+1}(e_1,\ldots;\alpha_1,\ldots,\alpha_{k+1})=\sum_{\mu=1}^{k+1}\omega_k(\delta\alpha_\mu,e_1,\ldots;
\alpha_1,\ldots,\widehat{\alpha}_\mu,\ldots,\alpha_{k+1})
$$
does not vanish in general. Nevertheless, this formula suggests a
fix. Let us define $\F_k\subset\C_k$ as consisting of
$\omega=(\omega_0,\omega_1,\ldots)$ such that, for each $i=1,\ldots,k$,
$\omega_i$ vanishes if any $k-i+1$ of its arguments are of the form
$\delta\alpha$ for some $\alpha\in\Omega^1$. Notice that, because of
\eqref{eqn:weakantisym}, it does not matter \emph{which} of the
arguments those are. Obviously, $\F_{k+1}\subset\F_k$.

\begin{prop}
$d\F_k\subset\F_k$ and $\F_{k_1}\F_{k_2}\subset\F_{k_1+k_2}$. In
particular, $\F_0$ is a differential graded subalgebra of
$\C(\E,\R)$ equal to $\pi^\vee\C(\bar{\E},\R)$.
\end{prop}
\begin{proof}
The first statement follows by inspection of formula
\eqref{eqn:diff'}, using Proposition \ref{prop:coanc-is-ideal}. For
the second one, suppose $\omega\in\F_{k_1}$, $\eta\in\F_{k_2}$, and
consider the expression \eqref{eqn:product} for $(\omega\eta)_k$:
$$
(\omega\eta)_k(e_1,\ldots)=
$$
$$=\sum_{i\geq 1}\sum_\sigma(-1)^\sigma\omega_i(e_{\sigma(1)},\ldots)
\eta_{k-i}(e_{\sigma(\deg\omega-2i+1)},\ldots)
$$
Suppose that $n=k_1+k_2-k+1$ of the arguments are $\delta\alpha$'s.
By our assumption, $\omega_i$ vanishes if at least $r=k_1-i+1$ of
its arguments are $\delta\alpha$'s, while $\eta_{k-i}$ vanishes if
at least $s=k_2-k+i+1$ of its arguments are $\delta\alpha$'s. Now,
in each term on the right hand side, some $m$ of the arguments of
$\omega_i$ are $\delta\alpha$'s, while the $n-m$ remaining
$\delta\alpha$'s are arguments of $\eta_{k-i}$. Since $r+s=n+1$,
either $m\geq r$ or $n-m\geq s$. Therefore, either $\omega_i$ or
$\eta_{k-i}$ vanishes; hence, so does $(\omega\eta)_k$.

The last statement is obvious.
\end{proof}

Let us also consider, for each $q>0$, the submodule
$\I^q\subset\C^q$ consisting of those
$\omega=(\omega_0,\omega_1,\ldots)$ such that for each $k$ and all
$\alpha_1,\ldots,\alpha_{q-2k}\in\Omega^1$,
$$\omega_k(\delta\alpha_1,\ldots,\delta\alpha_{q-2k})=0$$
Let $\I=\{I^q\}_{q>0}$.

\begin{prop}
$\I$ is a differential graded ideal of $\C(\E,\R)$.
\end{prop}

\begin{proof}
Follows immediately upon inspecting formulas \eqref{eqn:product} and
\eqref{eqn:diff'}, in view of Proposition \ref{prop:coanc-is-ideal}.
\end{proof}

We expect that the filtrations $\{\F_i\}$ and $\{\I^{(i)}\}$ (powers
of the ideal $\I$) will be useful in computing the cohomology of
$\E$ (see subsection \ref{subsec:naive} below).

\subsection{Some Cartan-like formulas} Given an $\alpha\in\Omega^1$,
consider the operator\\ $\iota_\alpha:\C\To\C[-2]$ defined by
$(\iota_\alpha\bar{\omega})_k=\iota_\alpha\bar{\omega}_{k+1}$, i.e
\begin{equation}\label{eqn:iotaalpha}
(\iota_\alpha\bar{\omega})_k(e_1,\ldots;\alpha_1,\ldots,\alpha_k)=\bar{\omega}_{k+1}(e_1,\ldots;\alpha,\alpha_1,\ldots,\alpha_k)
\end{equation}
It is easy to check that this defines a derivation of the algebra
$\C(\E,\R)$. For $f\in\R$, define the operator $\iota_f$ so that
\begin{equation}\label{eqn:iotaf}
\overline{\iota_f\omega}=\iota_{d_0f}\bar{\omega}
\end{equation}
Similarly, for any $e\in\E$, the operator $\iota_e:\C\To\C[-1]$
given by
\begin{equation}\label{eqn:iotae}
(\iota_e\omega)_k(e_2,\ldots)=\omega_k(e,e_2,\ldots)
\end{equation}
defines a derivation of $\C(\E,\R)$ of degree $-1$.

Recall that the $\KK$-module $L=\R[2]\oplus\E[1]$ forms a graded Lie
algebra with respect to the brackets $-\metr$.

\begin{prop}
The assignments $f\mapsto\iota_{f}$ and $e\mapsto\iota_e$ define an
action of the graded Lie algebra $L$ on $\C(\E,\R)$ by derivations.
\end{prop}
\begin{proof}
The only non-trivial commutation relation is
\begin{equation}\label{eqn:cartan1st}
\pb{\iota_{e_1}}{\iota_{e_2}}=\iota_{-\met{e_1}{e_2}}
\end{equation}
which follows immediately from \eqref{eqn:weakantisym}.
\end{proof}
Let us now define
\begin{equation}
L_e=\pb{\iota_e}{d}\quad \mathrm{and}\quad L_f=\pb{\iota_f}{d}
\end{equation}
Then the following analogues of the well-known Cartan commutation
relations hold:
\begin{eqnarray}
L_f&=&\iota_{\dl f}\\
\pb{L_f}{\iota_e}&=&\iota_{-\met{\dl
f}{e}}=\iota_{-\rho(e)f}=\pb{L_e}{\iota_f}\\
\pb{L_{e_1}}{\iota_{e_{2}}}&=&\iota_{[e_1,e_2]}\\
\pb{L_f}{L_g}&=&0\\
\pb{L_e}{L_f}&=&L_{\met{e}{\dl f}}=L_{\rho(e)f}=-\pb{L_f}{L_e}\\
\label{eqn:cartanlast}
\pb{L_{e_1}}{L_{e_2}}&=&L_{[e_1,e_2]}
\end{eqnarray}
We leave the derivation of these identities as an easy exercise for
the reader.

\begin{rem}
The assignment $\alpha\mapsto\iota_\alpha$ is $\R$-linear while
$f\mapsto\iota_f$ and $e\mapsto\iota_e$ are not. If $d_0\alpha=0$,
one has
$$L_\alpha=\pb{\iota_\alpha}{d}=\iota_{\delta\alpha}$$
but otherwise the algebra does not close. This is because there are
more derivations of $\C(\E,\R)$ of negative degree than we have
accounted for here: there are also derivations coming from maps
$\phi\in\Hom_\R(\E,\Omega^1)$, of the form
$$
(\iota_\phi\omega)_k(e_1,\ldots)=\sum_{i\geq
0}(-1)^{i-1}\iota_{\phi(e_i)}\omega_{k+1}(e_1,\ldots,\widehat{e}_i,\ldots)
$$
A description of the full algebra of derivations will be done in the
sequel.
\end{rem}

\section{Some applications}\label{sec:appl}

\subsection{$H^2$ and central extensions} Let us consider
extensions of $\R$-modules of the form
\begin{equation}\label{eqn:centrext}
\R\stackrel{i}{\rightarrowtail}\widehat{\E}\stackrel{p}{\twoheadrightarrow}\E
\end{equation}
\begin{defn}
Suppose $(\E,\metr,\dl,\brac)$ and
$(\widehat{\E},\metr',\dl',\brac')$ are Courant-Dorfman algebras and
$p:\widehat{\E}\To\E$ is a strict morphism fitting into
\eqref{eqn:centrext}. We say that \eqref{eqn:centrext} is a
\emph{central extension} of Courant-Dorfman algebras if the
following conditions hold:
\begin{enumerate}
\item $(i(f))^\flat=0$ for all $f\in\R$;
\item $[\hat{e},i(f)]=\rho'(\hat{e})f$ for all
$\hat{e}\in\widehat{\E}$ and $f\in\R$, where $\rho'$ is the anchor
of $\widehat{\E}$.
\end{enumerate}
A (necessarily iso) \emph{morphism} of central extensions is a
morphism of extensions \eqref{eqn:centrext} which is also a
Courant-Dorfman morphism.
\end{defn}

\begin{prop}
The $\KK$-module of isomorphism classes of central extensions
\eqref{eqn:centrext} which are split as metric $\R$-modules is
isomorphic to $H^2(\E,\R)$.
\end{prop}
\begin{proof}
The extension being split as metric $\R$-modules means that
$\widehat{\E}$ is isomorphic to $\E\oplus\R$ as an $\R$-module in
such a way that
\begin{equation}\label{eqn:extmetr}
\met{(e_1,f_1)}{(e_2,f_2)}'=\met{e_1}{e_2}
\end{equation}
The argument follows the well-known pattern: $\dl'$ necessarily has
the form
\begin{equation}\label{eqn:extder}
\dl'f=(\dl f,-\omega_1(f))
\end{equation}
for some $\omega_1\in\X^1$, while the bracket must have the form
\begin{equation}\label{eqn:extbrac}
[(e_1,f_1),(e_2,f_2)]'=([e_1,e_2],\rho(e_1)f_2-\rho(e_2)f_1+\omega_0(e_1,e_2))
\end{equation}
for some $\omega_0$ such that
$\omega=(\omega_0,\omega_1)\in\C^2(\E,\R)$; these define a
Courant-Dorfman structure if and only if $d\omega=0$. Conversely,
any 2-cocycle $\omega$ defines a Courant-Dorfman structure on
$\E\oplus\R$ by the formulas \eqref{eqn:extmetr}, \eqref{eqn:extder}
and \eqref{eqn:extbrac}. Furthermore, the central extensions given
by cocycles $\omega$ and $\omega'$ are isomorphic if and only if
$\omega-\omega'=d\lambda$ for a $\lambda\in\C^1(\E,\R)=\E^\vee$, the
isomorphism given by $\hat{e}\mapsto\hat{e}+i(\lambda(p(\hat{e})))$,
and conversely, every such $\lambda$ gives an isomorphism of
extensions. We leave it to the reader to check the details.
\end{proof}

\begin{eg}
Every closed $\omega\in\Omega^{2,\cl}$ gives rise to a central
extension of \emph{any} Courant-Dorfman algebra by the cocycle
$\rho^*\omega$ (\eqref{eqn:rhostar1},\eqref{eqn:rhostar2}).
\end{eg}

\subsection{$H^3$ and the canonical class} Given an almost Courant-Dorfman
algebra $\E$, consider the cochain
$\Theta=(\Theta_0,\Theta_1)\in\C^3(\E,\R)$ defined as follows:
\begin{eqnarray}\label{eqn:canon0}
\Theta_0(e_1,e_2,e_3)&=&\met{[e_1,e_2]}{e_3}\\
\Theta_1(e;f)&=&-\rho(e)f \label{eqn:canon1}
\end{eqnarray}
To see that $\Theta\in\C^3$, we need to verify relations
\eqref{eqn:weakantisym}:
$$
\Theta_0(e_1,e_2,e_3)+\Theta_0(e_2,e_1,e_3)=
$$
$$
=\met{[e_1,e_2]}{e_3}+\met{[e_2,e_1]}{e_3}=\rho(e_3)\met{e_1}{e_2}=-\Theta_1(e_3;\met{e_1}{e_2})
$$
and
$$
\Theta_0(e_1,e_2,e_3)+\Theta_0(e_1,e_3,e_2)=
$$
$$
=\met{[e_1,e_2]}{e_3}+\met{[e_1,e_3]}{e_2}=\rho(e_1)\met{e_2}{e_3}=-\Theta_1(e_1;\met{e_2}{e_3})
$$
are consequences of conditions (3) and (2) of Definition \eqref{def:DA},
respectively.
\begin{prop}\label{prop:dtheta} If $\E$ is a Courant-Dorfman algebra, $d\Theta=0$;
for any $\psi\in\Omega^{3,\mathrm{cl}}$, the Courant-Dorfman algebra
$\mathrm{Tw}(\psi)(\E)$ has
$$\Theta_\psi=\Theta+\rho^*\psi$$
\end{prop}
\begin{proof}
In fact, a computation using conditions (2) and (3) of Definition
\ref{def:DA} yields
\begin{eqnarray}\label{eqn:dtheta0}
(d\Theta)_0(e_1,e_2,e_3,e_4)&=&2\met{[e_1,[e_2,e_3]]-[[e_1,e_2],e_3]-[e_2,[e_1,e_3]]}{e_4}\\
\label{eqn:dtheta1}
(d\Theta)_1(e_1,e_2;f)&=&2\met{[\dl f,e_1]}{e_2}\\
\label{eqn:dtheta2}
-(d\Theta)_2(f_1,f_2)&=&2\met{\dl f_1}{\dl f_2}
\end{eqnarray}
Therefore, $d\Theta=0$ by conditions (4), (5) and (6) of Definition
\ref{def:DA}. The second statement follows immediately from the
formulas \eqref{eqn:twist}, \eqref{eqn:rhostar1} and
\eqref{eqn:rhostar2}.
\end{proof}
We shall call $\Theta$ the \emph{canonical cocycle} of $\E$ and its
class $[\Theta]\in H^3(\E,\R)$ -- the \emph{canonical class} of
$\E$.

\begin{rem}
If $i:\D\To\E$ is an isotropic submodule, $i^\vee\Theta$ is
$\R$-trilinear and alternating; if $\D$ is Dirac, $i^\vee\Theta=0$.
When $\metr$ is strongly non-degenerate and $\D$ is \emph{maximally}
isotropic, we can say ``and only if". This is the criterion
originally used by Courant and Weinstein \cite{CouWe} to define
Dirac structures in $\Q_0=\X^1\oplus\Omega^1$
(Example \ref{eg:exact}).
\end{rem}

\begin{eg}
For the ``original" Courant-Dorfman algebra $\Q_0$, we have $\Theta=d\omega$ where
\begin{eqnarray*}
\omega_0((v_1,\alpha_1),(v_2,\alpha_2))&=&\iota_{v_1}\alpha_2-\iota_{v_2}\alpha_1\\
\omega_1&=&0
\end{eqnarray*}
(proof left to the reader). Hence, the canonical class of $\Q_0$ is zero. It follows
that for any $\psi\in\Omega^{3,\cl}$ the canonical class of
$\Q_\psi$ is the image of $[\psi]\in H^3_{\mathrm{dR}}$.
\end{eg}

\section{The non-degenerate case}\label{sec:pb}
Let us now restrict attention to the special case of Courant-Dorfman
algebras which are non-degenerate in the sense of Definition
\ref{def:nondeg}.

\subsection{The Poisson bracket}

Recall that a strongly non-degenerate $\metr$ has an inverse
$$\pbr:\E^\vee\otimes_\R\E^\vee\To\R$$
defined by formula \eqref{eqn:inversemetr}. This operation can be
extended to a Poisson bracket on $\C=\C(\E,\R)$:
$$\pbr:\C\otimes\C\To\C[-2]$$
which we shall now define. Recall that, for an
$\omega=(\omega_0,\omega_1,\ldots)\in\C^p$, each $\omega_k$ is a
$\KK$-linear map
$$\omega_k:\E^{\otimes^{p-2k}}\To\X^k$$
which is $\R$-linear in the $(p-2k)$-th argument. Hence, by
adjunction, it gives rise to a $\KK$-linear map
$$\tilde\omega_k:\E^{\otimes^{p-2k-1}}\To\Hom_\R(S^k\Omega^1,\E^\vee)$$
defined as follows:
$$
\tilde\omega_k(e_1,\ldots,e_{p-2k-1})(f_1,\ldots,f_k)(e)=\omega_k(e_1,\ldots,e_{p-2k-1},e;f_1,\ldots,f_k)
$$
Define
$$
\omega_k^\sharp:\E^{\otimes^{p-2k-1}}\To\Hom_\R(S^k\Omega^1,\E)
$$
by $\omega_k^\sharp=(\tilde{\omega}_k)^\sharp$. Denote the inverse
of $(\cdot)^\sharp$ by $(\cdot)^\flat$.

\begin{rem}\label{rem:sharp}
These maps define an isomorphism (extending that of Definition
\ref{def:nondeg}) of graded $\R$-modules between $\C(\E,\R)$ and
$\C(\E,\E)$ whose elements are tuples $T=(T_0,T_1,\ldots)$ where
the maps
$$
T_k:\E^{\otimes^{p-2k-1}}\To\Hom_\R(S^k\Omega^1,\E)
$$
satisfy the conditions obtained by applying $(\cdot)^\sharp$ to
equations \eqref{eqn:weakantisym}; these make sense even when
$\metr$ is degenerate.
\end{rem}

Given $H\in\Hom_\R(S^i\Omega^1,\E)$, $K\in\Hom_\R(S^j\Omega^1,\E)$,
we can obtain $\langle H\cdot
K\rangle\in\Hom_\R(S^{i+j}\Omega^1,\R)$ by composing the product
\eqref{eqn:prodcan} in $\X$ with $\metr$, i.e.

$$
\langle H\cdot K\rangle(f_1,\ldots,f_{i+j})
=\sum_{\tau\in\sh(i,j)}\met{H(f_{\tau(1)},\ldots,f_{\tau(i)})}{K(f_{\tau(i+1)},\ldots,f_{\tau(i+j)})}
$$
\\
Now let $\omega=(\omega_0,\omega_1,\ldots)\in\C^p$,
$\eta=(\eta_0,\eta_1,\ldots)\in\C^q$. Let us define operations

$$\langle\omega\bullet\eta\rangle=(\langle\omega\bullet\eta\rangle_0,\langle\omega\bullet\eta\rangle_1,\ldots)$$
and
$$\omega\diamond\eta=((\omega\diamond\eta)_0,(\omega\diamond\eta)_1,\ldots)$$
with
$$
\langle\omega\bullet\eta\rangle_k,(\omega\diamond\eta)_k:\E^{\otimes^{p+q-2k-2}}\To\X^k
$$
given by the formulas
\begin{equation}\label{eqn:bullet}
\langle\omega\bullet\eta\rangle_k(e_1,\ldots,e_{p+q-2k-2})=
\end{equation}
$$
=(-1)^{q-1}\sum_{i+j=k}\sum_\sigma(-1)^\sigma
\langle\omega_i^\sharp(e_{\sigma(1)},\ldots,e_{\sigma(p-2i-1)})\cdot\eta_j^\sharp(e_{\sigma(p-2i)},\ldots,e_{\sigma(p+q-2k-2)})\rangle
$$
where $\sigma$ runs over $\sh(p-2i-1,q-2j-1)$, and
\begin{equation}\label{eqn:diamond}
(\omega\diamond\eta)_k(e_1,\ldots,e_{p+q-2k-2})=
\end{equation}
$$
=\sum_{i+j=k}\sum_\sigma(-1)^\sigma\omega_{i+1}(e_{\sigma(1)},\ldots,e_{\sigma(p-2i-2)})\circ\eta_j(e_{\sigma(p-2i-1)},\ldots,e_{\sigma(p+q-2k-2)})
$$
where $\sigma$ runs over $\sh(p-2i-2,q-2j)$, and $\circ$ in each
summand is defined as in \eqref{eqn:ins0}.

And finally, define
\begin{equation}\label{eqn:pb}
\pb{\omega}{\eta}=\omega\diamond\eta+\langle\omega\bullet\eta\rangle-(-1)^{pq}\eta\diamond\omega
\end{equation}

\begin{rem}
The subalgebra $\C_0=\Hom_\R(S(\E[1]),\R)$ is also closed under
$\pbr$; the restriction is given by
$$\pb{\omega_0}{\eta_0}=\langle\omega\bullet\eta\rangle_0$$
and in particular, for $\lambda,\mu\in\C^1$, the bracket reduces to
the formula \eqref{eqn:inversemetr}. At the other extreme, $\E=0$
(``vacuously non-degenerate"), we get
$\C=\Hom_\R(S_\R(\Omega^1[2]),\R)$, and the formulas
\eqref{eqn:product} and \eqref{eqn:pb} reduce, respectively, to the
classical formulas \eqref{eqn:prodcan} and \eqref{eqn:pbcan}.
\end{rem}

\begin{thm}\label{thm:pb}
Let $\E$ be a metric $\R$-module with a strongly non-degenerate
$\metr$.
\begin{itemize}
\item[(i)] The formula \eqref{eqn:pb} defines a non-degenerate Poisson
bracket on the algebra $\C(\E,\R)$ of degree $-2$;
\item[(ii)] For any almost Courant-Dorfman structure on $\E$, the
canonical cochain $\Theta$, defined by formulas \eqref{eqn:canon0}
and \eqref{eqn:canon1}, and the derivation $d$, defined by formula
\eqref{eqn:diff}, are related by
$$
d=-\pb{\Theta}{\cdot}
$$
\item[(iii)] The almost Courant-Dorfman structure is a Courant-Dorfman
structure if and only if
\begin{equation}\label{eqn:MC}
\pb{\Theta}{\Theta}=0
\end{equation}
\end{itemize}
\end{thm}
\begin{proof}
The first two statements are proved by a direct verification. The
``if" part of (iii) follows from (ii) and Proposition
\ref{prop:dtheta}, the ``only if" -- by formulas \eqref{eqn:dtheta0},
\eqref{eqn:dtheta1}, \eqref{eqn:dtheta2}, the nondegeneracy of
$\metr$ and the assumption that $\hf\in\KK$.
\end{proof}

\begin{rem}
Observe that $\brac$ and $\dl$ can be recovered from $\Theta$ via
\begin{eqnarray*}
[e_1,e_2]&=&\Theta_0^\sharp(e_1,e_2)\\
-\dl f&=&\Theta_1^\sharp(f)
\end{eqnarray*}
So the Poisson bracket $\pbr$ defines the differential graded Lie
algebra controlling the deformation theory of Courant-Dorfman
algebras with fixed underlying metric module with non-degenerate
$\metr$. In fact, we can use $(\cdot)^\sharp$ to lift $\pbr$ to
$\C(\E,\E)$ (Remark \ref{rem:sharp}) and obtain an explicit description of this bracket which
makes sense even if $\metr$ is degenerate. This is similar to the
description of the deformation complex of a Lie algebroid by Crainic
and Moerdijk \cite{CraMoe}. We shall postpone writing down these
formulas until the sequel to this paper, dealing with modules and
deformation theory.

We should mention, however, that, under the assumptions that $\E$ be projective and finitely
generated, and $\metr$ be full and strongly non-degenerate, the complex $\C(\E,\E)$ was already considered by Keller and Waldmann \cite{KelWald} who obtained for it a result (Theorem 3.17 of \emph{loc. cit.}) equivalent to our Theorem \ref{thm:pb} for $\C(\E,\R)$. The additional assumptions considerably reduce the amount of necessary calculations, in view of Remarks \ref{rem:full} and \ref{rem:fingen}.
\end{rem}

\subsection{The canonical class as obstruction to re-scaling} The
canonical class $[\Theta]$ has a familiar deformation-theoretic
interpretation. Let $t$ be a formal variable, and extend everything
$\KK[[t]]$-linearly to $\R[[t]]$, $\E[[t]]$. If $\Theta$ satisfies
the Maurer-Cartan equation \eqref{eqn:MC} and thus defines a
Courant-Dorfman structure on $(\R,\E)$, so does $\Theta_t=e^t\Theta$
on $(\R[[t]],\E[[t]])$. The question is, when is $\Theta_t$
isomorphic to $\Theta$? ``Isomorphic" here means that there exists an
automorphism $\phi(t)$ of the Poisson algebra $\C[[t]]$ with
$\phi(0)=\mathrm{id}$, and whose infinitesimal generator is
Hamiltonian with respect to an $\omega(t)\in\C^2[[t]]$, such that
$$
\phi(t)\Theta=\Theta_t
$$
Differentiating at $t=0$ immediately yields
$$
d\omega(0)=\Theta
$$
so in particular $[\Theta]=0$. Conversely, if this is the case,
$\phi(t)=\exp(t\pb{\omega(0)}{\cdot})$ does the trick.

If $\KK\supset\RR$, we can ask the same question for $t$ a real
number, rather than a formal variable. In this case, the condition
$[\Theta]=0$ is still necessary but not sufficient unless there
exists an $\omega(t)$ that integrates to a flow.

\subsection{Cartan relations and iterated brackets}
The following is easily verified:
\begin{prop}
Given $f\in\R$, $e\in\E$,
\begin{eqnarray*}
-\iota_f&=&\pb{f}{\cdot}\\
\iota_e&=&\pb{e^\flat}{\cdot}
\end{eqnarray*}
where $\iota_f$ and $\iota_e$ are given by \eqref{eqn:iotaf} and
\eqref{eqn:iotae}. Thus, the equations \eqref{eqn:cartan1st} --
\eqref{eqn:cartanlast} express commutation relations among
Hamiltonian derivations of $\C(\E,\R)$, analogous to the well-known
Cartan relations among derivations of $\Omega_\R$.
\end{prop}

\begin{cor}\label{cor:higherder}
For any $\omega=(\omega_0,\omega_1,\ldots)\in\C^p(\E,\R)$, the
following relation holds:
\begin{equation}\label{eqn:higherder}
\omega_k(e_1,\ldots,e_{p-2k};f_1,\ldots,f_k)=
\end{equation}
$$
=(-1)^{\frac{(p-2k)(p-2k-1)}{2}}\{\cdots\{\omega,e_1^\flat\},\cdots\},e_{p-2k}^\flat\},f_1\},\cdots\},f_k\}
$$
\end{cor}

\subsection{Relation with graded symplectic
manifolds}\label{subsec:grsymp} In this subsection we follow the
notation and terminology of \cite{Roy4-GrSymp}. Let $M_0$ be a
finite-dimensional $C^\infty$ manifold, $E\To M_0$ a vector bundle
of finite rank equipped with a pseudometric $\metr$. Consider the
isometric embedding
\begin{eqnarray*}
j:E&\To&E\oplus E^*\\
e&\longmapsto&(e,\hf e^\flat)
\end{eqnarray*}
with respect to the canonical pseudometric on $E\oplus E^*$,
inducing an embedding of graded manifolds
\begin{equation*}
j[1]:E[1]\To(E\oplus E^*)[1]
\end{equation*}
Define $M=M(E)$ to be the pullback of
$$T^*[2]E[1]\stackrel{p}{\To}(E\oplus E^*)$$
along $j[1]$, and let $\Xi\in\Omega^2(M)$ be the pullback of the
canonical symplectic form on $T^*[2]E[1]$. This $\Xi$ is closed, has
degree +2 with respect to the induced grading, and is non-degenerate
if and only if $\metr$ is, in which case its inverse gives a Poisson
bracket on the algebra $\C(M)$ of polynomial functions on $M$, of
degree -2. Conversely, we proved in \cite{Roy4-GrSymp} that every
degree-two graded symplectic manifold is isomorphic to $M(E)$ for
some $E$.

\begin{thm}\label{thm:isom_gr_man}
Let $\R=C^\infty(M_0)$, $\E=\Gamma(E)$. The map
$$\Phi:\C(M(E))\To\C(\E,\R)$$
given, for $\omega\in\C^p(M(E))$, by
$$
(\Phi\omega)_k(e_1,\ldots,e_{p-2k};f_1,\ldots,f_k)=
$$
$$
=(-1)^{\frac{(p-2k)(p-2k-1)}{2}}\{\cdots\{\omega,e_1^\flat\},\cdots\},e_{p-2k}^\flat\},f_1\},\cdots\},f_k\}
$$
is an isomorphism of graded Poisson algebras.
\end{thm}
\begin{proof}
That $\Phi$ takes values in $\C(\E,\R)$ (i.e. the relations
\eqref{eqn:weakantisym} hold) is a consequence of the Jacobi
identity for $\pbr$ and \eqref{eqn:inversemetr}. That $\Phi$ is a
map of Poisson algebras follows by applying Lemma \ref{lemma:der} below to
$\star$ being first the product and then the Poisson bracket on
$\C(M(E))$. The injectivity of $\Phi$ amounts to the statement that
$\omega$ is uniquely determined by the functions $(\Phi\omega)_k$,
$k=0,1,\ldots,[\deg\omega/2]$; this is most easily seen in local
coordinates where these functions are just the Taylor coefficients
of $\omega$. Surjectivity is a consequence of Corollary
\ref{cor:higherder}.
\end{proof}

\begin{lem}\label{lemma:der}
Let $A$ be a $\KK$-module equipped with a bilinear operation
$\star:A\otimes A\To A$; let $D_1,\ldots,D_k:A\To A$ be $\KK$-linear derivations of $\star$, and let $D=D_1\cdots
D_n$. Then
$$
D(a\star b)=\sum_{i+j=k}\sum_{\sigma\in\sh(i,j)}(D_{\sigma(1)}\cdots
D_{\sigma(i)}a)\star(D_{\sigma(i+1)}\cdots D_{\sigma(k)}b)
$$
\end{lem}
\begin{proof}
Induction.
\end{proof}
If $A$ and the $D$'s are graded, the lemma holds with appropriate
Koszul signs put in place.

\subsection{Relation with ``naive cohomology"}\label{subsec:naive} Let $\E$ be a
Courant-Dorfman algebra, $\K=\ker\rho$,
$\bar{\E}=\E/\delta\Omega^1$. The map $(\cdot)^\flat:\E\to\E^\vee$
extends to
$$
(\cdot)^\flat:\Lambda_\R\K\To\C(\E,\R)
$$
whose image is actually contained in $\F_0$, in view of
\eqref{eqn:rho-delta}. For $\R=C^\infty(M_0)$, $\E=\Gamma(E)$ and
$\metr$ non-degenerate, this map is an isomorphism onto $\F_0$,
which in turn is isomorphic to $\C(\bar{\E},\R)$ (see Subsection
\ref{subsec:filtrF}). Sti\'{e}non and Xu \cite{StXu} defined a
differential on the algebra $\Lambda_\R\K$ (in view of this
isomorphism, it is just the standard differential for the
Lie-Rinehart algebra $\bar{\E}$) and called its cohomology the
``naive cohomology" of the Courant algebroid $E$. They conjectured
that, if $\rho$ is surjective, the inclusion
$$
\Phi^{-1}\circ(\cdot)^\flat:\Lambda_\R\K\To\C(M(E))
$$
is a quasi-isomorphism. This was proved by Ginot and Grutzmann
\cite{GinGru} who also obtained further results by considering the
spectral sequence associated to the filtration of $\C(M(E))$ by the
powers of what they called ``the naive ideal". This ideal corresponds
under $\Phi$ to the ideal $\I$ we defined in Subsection
\ref{subsec:filtrF}.

\begin{rem}
For general $(\R,\E,\metr)$ it is not known (and probably false)
that the image of $\Lambda_\R\K$ in $\C(\E,\R)$ is closed under $d$.
\end{rem}

\section{Concluding remarks, speculations and open
ends}\label{sec:concl}
In conclusion, let us mention a few important issues we have not
touched upon here, which we plan to address in a sequel (or sequels)
to this paper.

\subsection{The pre-symplectic structure} The algebra $\C(\E,\R)$
has an extra structure: a closed 2-form $\Xi\in\Omega^2_\C$ which
has degree 2 with respect to induced grading, and is $d$-invariant
in the sense that
\begin{equation}\label{eqn:Xi-d-inv}
L_d\Xi=0
\end{equation}
where $L$ is the Lie derivative operator on $\Omega_\C$. This
two-form exists on general principles: for strongly non-degenerate
$\metr$ it is just the inverse of the Poisson tensor \eqref{eqn:pb},
while for $\R=C^\infty(M_0)$ and $\E=\Gamma(E)$ the construction
from \cite{Roy4-GrSymp} yields $\Xi$ for an arbitrary $\metr$ (see
subsection \ref{subsec:grsymp} for a review). The formulas
\eqref{eqn:Xi} define the induced bilinear form on the tangent
complex $\T_\E$; its $\delta$-invariance \eqref{eqn:Xi-delta-inv} is
just the linearization of \eqref{eqn:Xi-d-inv}. By Dirac's formalism
\cite{Dir} adapted to the graded setting, the closed 2-form $\Xi$
induces a Poisson bracket on a certain subalgebra $\C_\flat(\E,\R)$
of $\C(\E,\R)$.

\subsection{Morphisms}
The functors we have constructed,
\begin{eqnarray*}
\mathbf{Met}_\R&\To&\mathbf{gra}_\R^{\mathrm{op}}\\
(\E,\metr)&\longmapsto&\C(\E,\R)
\end{eqnarray*}
and
\begin{eqnarray*}
\mathbf{CD}_\R&\To&\mathbf{dga}_\R^{\mathrm{op}}\\
(\E,\metr,\dl,\brac)&\longmapsto&(\C(\E,\R),d)
\end{eqnarray*}
are not fully faithful for two reasons. The first has to do with
infinite dimensionality issues: not all maps $\F^\vee\to\E^\vee$
come from maps $\E\to\F$, duals of tensor products are not tensor
products of duals, and so on. These issues can be dealt with by
calling those maps of duals which are duals of maps
\emph{admissible} and restricting attention only to such maps; one
can similarly define admissible derivations, and so on. Of course,
this only makes sense for objects in the image of the above functors.

However, even if we restrict attention to the finite-dimensional and
locally free case, the functors above are still not full. This is
because we have only defined \emph{strict} maps of Courant-Dorfman
algebras; the more general notion of a \emph{lax} map can be
obtained as admissible dg map preserving $\Xi$ in an evident way;
this way we can also describe maps of Courant-Dorfman algebras over
different base rings.

Finally, we have defined (strict) morphisms from Lie-Rinehart to
Courant-Dorfman algebras and back, but no category containing both
kinds of algebras as objects. This problem can be solved by
introducing ``Lie-Rinehart 2-algebras" (algebraic analogues of Lie
2-algebroids) and their weak (and maybe also higher) morphisms,
which can again be reduced to studying dg algebras of a certain kind
and admissible dg morphisms between them.

\subsection{Modules} We have not defined the notion of a module over
a Courant-Dorfman algebra and cohomology with coefficients, except
in the trivial module $\R$. Again, this can be done by analyzing
(the derived category of) dg modules over the dga $\C(\E,\R)$ and
trying to describe them explicitly in terms of $\E$. It is not clear
though what, if any, compatibility with $\Xi$ we should require.

\subsection{The Courant-Dorfman operad.}
The infinite-dimensionality problems mentioned above arise because
our construction of the algebra $\C(\E;\R)$ involves dualization. It
seems more natural to try to construct some sort of coalgebra
instead. In operad theory, Koszul duality provides a systematic way
of obtaining such a differential graded coalgebra from an algebra
over a given \emph{quadratic} operad. There is a an operad, $\C\D$,
on the set of two colors, whose algebras are Courant-Dorfman
algerbras; as operads go, this is a pretty nasty one: inhomogeneous
cubic, so Koszul duality does not apply. However, if $\metr$ is
non-degenerate, one can replace $\dl$ by an action of $\E$ on $\R$
via the anchor $\rho$ and get rid of the offending relations, ending
up with an algebra over a nice homogeneous quadratic operad (this is
actually the formulation given in \cite{Roy4-GrSymp}). Of course,
non-degeneracy is not a condition that can be expressed in operadic
terms; more importantly, even if we ignore this and try to apply
Koszul duality to the resulting quadratic operad, we will get a
wrong answer, because we are really interested in Courant-Dorfman
structures over a \emph{fixed} underlying metric module (which we
can assume to be non-degenerate if we want to). What is relevant in
this situation (which also arises in several other contexts we know
of) is a kind of \emph{relative} deformation theory for algebras
over a \emph{pair} of operads $P\subset Q$, where we want to vary
the $Q$-algebra structure while keeping the underlying $P$-structure
fixed. As far as we are aware, such a theory is not yet available,
but it would be interesting and useful to try to develop it.

\appendix

\section{K\"{a}hler differential forms and multiderivations.}\label{app:kaehler}

Let $\R$ be a commutative $\KK$-algebra, $\M$ an $\R$-module. A
$\KK$-linear map
$$D:\R\To\M$$
is called a \emph{derivation} if it satisfies the Leibniz rule:
$$D(fg)=(Df)g+f(Dg)\quad\forall f,g\in\R$$
$\M$-valued derivations form an $\R$-module denoted
$\mathrm{Der}(\R,\M)$; the assignment is functorial in $\M$.

The functor $\M\mapsto\mathrm{Der}(\R,\M)$ is (co)representable:
there exists an $\R$-module $\Omega^1=\Omega^1_\R$, unique up to a
unique isomorphism, together with a natural (in $\M$) isomorphism of
$\R$-modules
$$\mathrm{Der}(\R,\M)\simeq\Hom_\R(\Omega^1,\M)$$
In particular, putting $\M=\Omega^1$, the identity map on the right
hand side corresponds to the universal derivation
$$d_0:\R\To\Omega^1$$
$\Omega^1$ is referred to as the \emph{module of K\"{a}hler
differentials}; it can be described explicitly as consisting of
formal finite sums of terms of the form $fd_0g$ with $f,g\in\R$,
subject to the Leibniz relation
$$d_0(fg)=(d_0f)g+fd_0g$$
The algebra of K\"{a}hler differential forms is obtained by taking
$\Omega=\{\Omega^k\}_{k\geq 0}$ with
$\Omega^k=\Lambda_\R^k\Omega^1$. It is associative and
graded-commutative with respect to exterior multiplication. The
universal derivation $d_0$ extends to an odd derivation of $\Omega$
satisfying $d_0^2=0$, called the \emph{de Rham differential}, or the
\emph{exterior derivative}. The algebra of K\"{a}hler differential
forms is the universal differential algebra containing $\R$.

The module $\X^1=\mathrm{Der}(\R,\R)$ forms a Lie algebra under the
commutator bracket $\{\cdot,\cdot\}$. By the universal property of
$\Omega^1$ one has
$$\X^1\simeq\Hom_\R(\Omega^1,\R)=(\Omega^1)^\vee$$
Given $v\in\X^1$, we denote the corresponding operator on the right
hand side by $\iota_v$. It extends to a unique odd derivation of
$\Omega$, denoted by the same symbol. The Lie derivative operator is
defined by the Cartan formula
$$L_v=\pb{d_0}{\iota_v}$$
The operators $\iota_v$, $L_v$ and $d_0$ are subject to the usual
Cartan commutation relations
$$\pb{\iota_v}{\iota_w}=0;\quad\pb{L_v}{\iota_w}=\iota_{\{v,w\}};\quad\pb{L_v}{L_w}=L_{\{v,w\}},$$
describing an action of the differential graded Lie algebra
$T[1]\X^1=\X^1[1]\oplus\X^1$ on $\Omega$.

K\"{a}hler differential forms should be distinguished from the usual
differential forms
$\widetilde{\Omega}_\R=\{\widetilde{\Omega}^k\}_{k\geq 0}$ on $\R$,
where $\widetilde{\Omega}^k$ is defined as the module of alternating
$k$-multilinear functions on $\X^1$:
$$\widetilde{\Omega}^k=\Hom_\R(\Lambda_\R\X^1,\R)$$
Of course, one has the canonical inclusion
$$\Omega=\Lambda_\R\Omega^1\hookrightarrow(\Lambda_\R(\Omega^1)^\vee)^\vee=\widetilde{\Omega}$$
which generally fails to be an isomorphism unless $\R$ satisfies
certain finiteness conditions. Nevertheless, the exterior
multiplication and differential $d_0$ extend to $\widetilde{\Omega}$
and are defined by the usual Cartan formulas.

Let $\X^0=\R$ and, for $k>0$, let $\X^k$ denote the $\R$-module of
symmetric $k$-derivations of $\R$, that is, symmetric $k$-linear
forms (over $\KK$) on $\R$ with values in $\R$ which are derivations
in each argument. Again, by abstract nonsense we have
$$\X^k\simeq\Hom_\R(S_\R^k\Omega^1,\R)$$
The function on the right hand side corresponding to a
$k$-derivation $H$ on the left will be denoted by $\bar{H}$, so that
$$H(f_1,\ldots,f_k)=\bar{H}(d_0f_1,\ldots,d_0f_k).$$
The graded module  of symmetric multi-derivations,
$\X=\{\X^k\}_{k\geq 0}$, forms a graded commutative algebra over
$\R$ (if we assign to elements of $\X^k$ degree $2k$); the
multiplication is given by the following explicit formula:
\begin{equation}\label{eqn:prodcan}
HK(f_1,\ldots,f_{i+j})=\sum_{\tau\in\mathrm{sh}(i,j)}H(f_{\tau(1)},\ldots,f_{\tau(i)})K(f_{\tau(i+1)},\ldots,f_{\tau(i+j)})
\end{equation}
Furthermore, $\X$ has a natural Poisson bracket, extending the
commutator of derivations and the natural action of $\X^1$ on $\R$;
it is given by the formula:
\begin{equation}\label{eqn:pbcan}
\{H,K\}=H\circ K-K\circ H
\end{equation} where
\begin{equation}\label{eqn:ins0}
H\circ K(f_1,\ldots,f_{i+j-1})
=\sum_{\tau\in\mathrm{sh}(i,j-1))}H(K(f_{\tau(1)},\ldots,f_{\tau(i)}),f_{\tau(i+1)},\ldots,f_{\tau(i+j-1)})
\end{equation}
for $H\in\X^i$, $K\in\X^j$. This Poisson bracket has degree -2 with
respect to the grading just introduced.

Given an $\alpha\in\Omega^1$, denote by $\iota_\alpha$ the evident
contraction operator on $\X$. It is a derivation of the
multiplication, but not of the Poisson bracket, unless
$d_0\alpha=0$.

\section{Lie-Rinehart algebras.}\label{app:LR}
\begin{defn}
A Lie-Rinehart algebra consists of the following data:
\begin{itemize}
\item a commutative $\KK$-algebra $\R$;
\item an $\R$-module $\L$;
\item an $\R$-module map $\rho:\L\To\X^1=\Der(\R,\R)$, called the \emph{anchor};
\item a $\KK$-bilinear Lie bracket
$\brac:\L\otimes\L\To\L$.
\end{itemize}
These data are required to satisfy the following additional
conditions:
\begin{enumerate}
\item $[x_1,fx_2]=f[x_1,x_2]+(\rho(x_1)f)x_2$;
\item $\rho([x_1,x_2])=\pb{\rho(x_1)}{\rho(x_2)}$
\end{enumerate}
for all $x_1,x_2\in\L$, $f\in\R$.

A \emph{morphism} of Lie-Rinehart algebras over $\R$ is a map of the
underlying $\R$-modules commuting with anchors and brackets in an
obvious way. Lie-Rinehart algebras over $\R$ form a category denoted
by $\mathbf{LR}_\R$.
\end{defn}

\begin{eg}
$\X^1=\Der(\R,\R)$ becomes a Lie-Rinehart algebra with respect to
the commutator bracket $\pbr$ and the identity map $\X^1\To\X^1$ as
the anchor. This is the terminal object in $\mathbf{LR}_\R$: the
anchor of each Lie-Rinehart algebra gives the unique map.
\end{eg}

\begin{eg}
Let $\M$ be an $\R$-module; a \emph{derivation of} $\M$ is a pair
$(D,\sigma)$, where $D:\M\To\M$ is a $\KK$-linear map and
$\sigma=\sigma_D\in\Der(\R,\M)$, satisfying the following
compatibility condition:
$$
D(fm)=fD(m)+\sigma(f)m
$$
Derivations of $\M$ form an $\R$-module which we denote $\Der(\M)$;
moreover, $\Der(\M)$ is a Lie-Rinehart algebra with respect to the
commutator bracket and the anchor $\pi$ given by the assignment
$(D,\sigma)\mapsto\sigma$.
\end{eg}
\begin{defn}
A \emph{representation} of a Lie-Rinehart algebra $\L$ on an
$\R$-module $\M$ is a map of Lie-Rinehart algebras
$\nabla:\L\To\Der(\M)$. In other words, $\nabla$ assigns, in an
$\R$-linear way, to each $x\in\L$ a derivation $(\nabla_x,\rho(x))$
such that
$$
\nabla_{[x,y]}=\pb{\nabla_x}{\nabla_y}
$$
An $\R$-module $\M$ equipped with a representation of $\L$ is said
to be an $\L$-module.
\end{defn}

\begin{eg}
For every Lie-Rinehart algebra $\L$, $\R$ is an $\L$-module with
$\nabla=\rho$.
\end{eg}

\begin{eg}
Let $\L$ be a Lie-Rinehart algebra and let $\K=\ker(\rho)$. Then
$\K$ becomes an $\L$-module with
$$
\nabla_x(y)=[x,y]
$$
for $x\in\L$, $y\in\K$. Moreover, $\K$ is a Lie algebra over $\R$
with respect to the restricted bracket, and $\nabla$ acts by
derivations of this bracket.
\end{eg}

Given an $\L$-module $\M$, one defines for each $q\geq 0$ the module
of $q$-cochains on $\L$ with coefficients in $\M$ to be
$$
\widetilde{\Omega}^q(\L,\M)=\Hom_\R(\Lambda_\R\L,\M).
$$
The differential
$d:\widetilde{\Omega}^q(\L,\M)\To\widetilde{\Omega}^{q+1}(\L,\M)$ is
given by the standard (Chevalley-Eilenberg-Cartan-de Rham) formula
\begin{equation}\label{eqn:d_LR}
d\eta(x_1,\ldots,x_{q+1})=\sum_{i=1}^{q+1}(-1)^{i-1}\nabla_{x_i}\eta(x_1,\ldots,\hat{x}_i,\ldots,x_{q+1})+
\end{equation}
$$
+\sum_{i<j}(-1)^{i+j}\eta([x_i,x_j],x_1,\ldots,\hat{x}_i,\ldots,\hat{x}_j,\ldots,x_{q+1})
$$

\begin{rem}
Notice that, for $\L=\X^1$ and $\M=\R$, this yields
$\widetilde{\Omega}_\R$, rather than $\Omega_\R$. It is possible
(and probably more correct in general) to consider the complex
$\Omega(\L,\M)$ with differential given by the universal property of
the K\"{a}hler forms.
\end{rem}

\begin{rem}
The term ``Lie-Rinehart algebra" is due to J. Huebschmann
\cite{Hueb-PoisCoh}, and is based on the work of G.S. Rinehart who
studied these structures in a seminal paper \cite{Rinehart}
(although Rinehart himself referred to earlier work of Herz and
Palais).
\end{rem}

\section{Leibniz algebras, modules and
cohomology.}\label{app:Leibniz}

This section follows Loday and Pirashvili \cite{LodPir} closely. A
\emph{Leibniz algebra} over $\KK$ is a $\KK$-module $\E$ equipped
with a bilinear operation
$$\brac:\E\otimes\E\To\E$$
satisfying the following version of the Jacobi identity:
$$[e_1,[e_2,e_3]]=[[e_1,e_2],e_3]+[e_2,[e_1,e_3]]$$
(i.e., $[e,\cdot]$ is a derivation\footnote{In fact, this defines a
\emph{left} Leibniz algebra, whereas Loday and Pirashvili considered
\emph{right} Leibniz algebras, in which $[\cdot,e]$ is a right
derivation of $\brac$. The assignment $\brac\To\brac^{\mathrm{op}}$
where
$$
[x,y]^{\mathrm{op}}=-[y,x]
$$
establishes an isomorphism of the categories of these two kinds of
Leibniz algebras; the formulas for modules and differentials have to
be modified accordingly.} of $\brac$ for each $e\in\E$).

Given a Leibniz algebra $\E$, an $\E$-\emph{module} is a
$\KK$-module $\M$ equipped with two structure maps: a left action
\begin{eqnarray*}
\E\otimes\M&\To&\M\\
(e,m)&\mapsto&[e,m]
\end{eqnarray*}
and a right action
\begin{eqnarray*}
\M\otimes\E&\To&\M\\
(m,e)&\mapsto&[m,e]
\end{eqnarray*}
satisfying the following equations:
\begin{eqnarray*}
{[e_1,[e_2,m]]}&=&[[e_1,e_2],m]+[e_2,[e_1,m]]\\
{[e_1,[m,e_2]]}&=&[[e_1,m],e_2]+[m,[e_1,e_2]]\\
{[m,[e_1,e_2]]}&=&[[m,e_1],e_2]+[e_1,[m,e_2]]
\end{eqnarray*}
Maps of $\E$-modules are defined in an obvious way.

Given any Leibniz algebra $\E$, a left $\E$-action on $\M$
satisfying the first of the above three equations can be extended to
an $\E$-module structure in two standard ways, by defining the right
action either by
$$[m,e]:=-[e,m]$$
or by
$$[m,e]=0$$
Following Loday and Pirashvili, we call the first one
\emph{symmetric}, the second -- \emph{antisymmetric}.

Given a Leibniz algebra $\E$ and an $\E$-module $\M$, define the
complex of cochains on $\E$ with values in $\M$ by setting, for
$q\geq 0$,
$$
\C^q_{\mathrm{LP}}(\E,\M)=\Hom(\E^{\otimes^q},\M)
$$
with the differential
$$
d_{\mathrm{LP}}:\C^q_{\mathrm{LP}}(\E,\M)\To\C^{q+1}_{\mathrm{LP}}(\E,\M)
$$
given by
\begin{equation}\label{eqn:d_LP}
d_{\mathrm{LP}}\eta(e_1,\ldots,e_{q+1})=\sum_{i=1}^q(-1)^{i-1}[e_i,\eta(\ldots,\hat{e}_i,\ldots)]+
\end{equation}
$$
+(-1)^{q+1}[\eta(e_1,\ldots,e_q),e_{q+1}]+\sum_{i<j}(-1)^i\eta(e_1,\ldots,\hat{e}_i,\ldots,\hat{e}_j,[e_i,e_j],e_{j+1},\ldots,e_{q+1})
$$
If the module $\M$ is symmetric, this reduces to
\begin{equation}\label{eqn:d_LP_sym}
d_{\mathrm{LP}}\eta(e_1,\ldots,e_{q+1})=\sum_{i=1}^{q+1}(-1)^{i-1}[e_i,\eta(\ldots,\hat{e}_i,\ldots)]+
\end{equation}
$$
+\sum_{i<j}(-1)^i\eta(e_1,\ldots,\hat{e}_i,\ldots,\hat{e}_j,[e_i,e_j],e_{j+1},\ldots,e_{q+1})
$$

\providecommand{\bysame}{\leavevmode\hbox to3em{\hrulefill}\thinspace}
\providecommand{\MR}{\relax\ifhmode\unskip\space\fi MR }
\providecommand{\MRhref}[2]{%
  \href{http://www.ams.org/mathscinet-getitem?mr=#1}{#2}
}
\providecommand{\href}[2]{#2}

\end{document}